\documentclass[dvips,preprint]{imsart}
\usepackage[latin1]{inputenc}
\usepackage{amsmath, amssymb, amsthm}
\usepackage[dvips]{graphicx}

\startlocaldefs

\newcommand{\cov}{\operatorname{Cov}}
\newcommand{\var}{\operatorname{Var}}

\newcommand{\com}{\operatorname{Com}}
\newcommand{\Id}{\operatorname{Id}}

\newcommand{\Tr}{\mathop{\rm Tr}}
\newcommand{\card}{\mathop{\rm Card}}
\newcommand{\re}{\operatorname{Re}}
\newcommand{\id}{\mbox{\rm1\hspace{-.2ex}\rule{.1ex}{1.44ex}}
   \hspace{-.82ex}\rule[-.01ex]{1.07ex}{.1ex}\hspace{.2ex}}
\newtheorem{theo}{Theorem}[section]
\newtheorem{pr}{Proposition}[section]

\newtheorem{lem}{Lemma}[section]

\endlocaldefs

\begin{document}

\begin{frontmatter}

\title{On quenched and annealed critical curves of random pinning model with finite range correlations}
\runtitle{Pinning with correlated disorder}

\author{\fnms{Julien} \snm{Poisat} \ead[label=e1]{poisat@math.univ-lyon1.fr}}
\address[label=e2]{Institut Camille Jordan\\ 43 bld du 11 novembre 1918 \\ 69622 Villeurbanne, France\\ Tel.: +33(0)472.44.79.41\\ \printead{e1}}
\affiliation{Universit\'e Lyon 1}

\runauthor{Julien Poisat}

\begin{abstract}
This paper focuses on directed polymers pinned at a disordered and correlated interface. We assume that the disorder sequence is a q-order moving average and show that the critical curve of the annealed model can be expressed in terms of the Perron-Frobenius eigenvalue of an explicit transfer matrix, which generalizes the annealed bound of the critical curve for i.i.d. disorder. We provide explicit values of the annealed critical curve for $q=1$,$2$ and a weak disorder asymptotic in the general case. Following the renewal theory approach of pinning, the processes arising in the study of the annealed model are particular Markov renewal processes. We consider the intersection of two replicas of this process to prove a result of disorder irrelevance (i.e. quenched and annealed critical curves as well as exponents coincide) via the method of second moment.
\end{abstract}



\end{frontmatter}

\section{Introduction}

Polymers are macromolecules which are modelized by self-avoiding or directed random walks. Take for instance $S=(S_n)_{n\geq0}$ a random walk on $\mathbb{Z}$ starting at 0 and such that $|S_{n+1}-S_n|\leq1$. By polymer of dimension 1+1 and size N we will mean a realization of the directed random walk $\left\{(n,S_n) \right\}_{0\leq n\leq N}$, where each segment $\left[(n,S_n), (n+1,S_{n+1}) \right]$ stands for a constitutive unit, called monomer.

Suppose now that a reward $h$ is given to a configuration $\left\{(n,S_n) \right\}_{0\leq n\leq N}$ each time it touches the interface, i.e. each time $S_n=0$. One can then consider a distribution on polymers of size N whose density with respect to the initial distribution is equal, up to a renormalizing constant, to the Boltzmann factor
\[\exp\left( h \times \card\{n\in\{1,\ldots,N\}|S_n=0\} \right).\]
Depending on the sign of $h$, this distribution favorizes or penalizes polymers pinned to the interface, and letting $N$ go to infinity, the model, called homogeneous pinning model, undergoes a localization/delocalization transition.

Pinning models can also be used to study the interaction between two polymers, since the difference of two random walks is still a random walk. One can think for example of the two complementary strands of a DNA molecule: in this case, the values of $n$ for which $S_n = 0$ are the sites where the two strands are pinned, and the delocalization transition corresponds to DNA denaturation (or melting). One could argue that the binding strength between the two strands actually depends on the base pair, which is A-T or G-C. This corresponds to looking at a disordered model, i.e. a model in which the reward is $n$-dependent. An assumption usually made is that the reward at site $n$ writes
\[h_n = h + \beta \omega_n\]
where $h\in\mathbb{R}$, $\beta \geq 0$ and $\omega=(\omega_n)_{n\geq0}$ is a frozen realization of a sequence of independent standard gaussian random variables. The space of parameters is then partitioned in localized and delocalized phases, separated by a critical curve $\beta \mapsto h_c(\beta)$. The presence of disorder has important consequences on the model. For example, one can show that there is localization for $h<0$ provided that disorder is strong enough (i.e. $\beta$ large enough). If we consider the annealed model (i.e. the model in which the Boltzmann factor is averaged over disorder), we have the following lower bound:
\begin{equation}\label{AB0}
 h_c(\beta) \geq -\log P(\tau_1 < +\infty) - \frac{\beta^2}{2}
\end{equation}
where $\tau_1$ is the first return time of $S$ to 0. In the last few years, many rigorous results were given on relevance of disorder, which in particular answer the following question: when is (\ref{AB0}) an equality? For these questions, as well as classical results on homogeneous and disordered pinning models, we refer to \cite{GG_Book}, \cite{GG_Overview}, \cite{Toninelli_Survey} and references therein.

In this paper we remove the independence assumption on $\omega$ and study the effect of correlations on the right-hand side of (\ref{AB0}), i.e on the annealed critical curve. This is partly motivated by the long-range correlations in DNA sequence, see \cite{Chen} and \cite{Peng} on this topic. We also mention \cite{UnzipDNA} and \cite{bubble} where the effect of sequence correlation is investigated, in somewhat different contexts. In \cite{UnzipDNA}, the authors study the effect of a pulling force applied to the extremity of a DNA strand on the number of broken base pairs (unzipping of DNA) in two correlated scenarii: integrable and nonintegrable correlations. In \cite{bubble}, the authors consider the effect of sequence correlation on the bubble size distribution: by bubbles we mean broken base pairs, and if we keep in mind the analogy with pinning models, it corresponds to the excursions of the directed random walk between two visits at $0$.

The disorder sequence in our model is a finite-order moving average of an i.i.d sequence, which is the simplest correlated sequence one can look at, and the reason for this choice will be clearer further in the text. This will be defined in Section \ref{model}, as well as the renewal sequence $\tau=(\tau_n)_{n\geq 0}$
(the contact points) and the polymer measures. In Section \ref{generalities}, we introduce classical notions for these models: the free energy, the phase diagram and the (quenched and annealed) critical curve of the model. In the proof of Theorem \ref{annealed_f}, a new homogeneous model emerges, whose hamiltonian does not only depends on the number of renewal points but also on their mutual distances. In Section \ref{results} we are interested in the annealed critical curve. The main results are Theorem \ref{annealed_critical_curve}, which states that the difference between the annealed critical curve in the correlated case and the annealed critical curve in the i.i.d. case can be expressed in terms of the Perron-Frobenius eigenvalue of an explicit transfer matrix, and Proposition \ref{asymptotic}, which gives a weak disorder asymptotic of the annealed critical curve. Note that the appearance of Perron-Frobenius eigenvalues is reminiscent of results on periodic copolymers, see \cite{CaGiZa07}. In a second part of the paper (Section \ref{irr_section}, Theorem \ref{result}), we show that under certain conditions (the same as i.i.d. disorder actually) quenched and critical curves (as well as exponents) coincide at high temperatures (small $\beta$). This is the regime of disorder irrelevance. We use the second moment method, which will lead us to study the exponential moments of two replicas of a certain Markov renewal process.

\section{The model}\label{model}
\subsection{Contact points between the polymer and the line}\label{def_renewal} We follow the renewal theory approach of pinning. Let $\tau$ be a discrete renewal process such that $\tau_0 = 0$ and $\tau_n = \sum_{k=1}^{n} T_k$, where the inter-arrival times (or jumps) $T_k$ are i.i.d. random variables taking values in $\overline{\mathbb{N}}^*$. Furthermore, $K(n) = P(T_1 = n) = \frac{L(n)}{n^{1+\alpha}}$ where $\alpha \geq 0$ and $L$ is a slowly varying function. Without losing in generality, we can assume that $\sum_{n\geq1} K(n) = 1$, i.e. $\tau$ is recurrent. We distinguish between positive recurrence ($\alpha >1$ or $\alpha =1$ and $L$ is such that $\sum_{n\geq1} L(n)/n < +\infty$) and null recurrence ($\alpha \in [0,1)$ or $\alpha =1$ and $L$ is such that $\sum_{n\geq1} L(n)/n = +\infty$). We will denote by $\delta_n$ the indicator of the event $\{n\in\tau\} = \bigcup_{k\geq0}\{\tau_k = n\}$ so that if $\imath_N := \sup \{k\geq0 | \tau_k \leq N\}$ is the number of renewal points before $N$, then $\imath_N = \sum_{n=1}^N \delta_n$. The letter $E$ will denote expectation with respect to the renewal process.

We also suppose that for all $n\geq1$, $K(n)>0$ (which implies aperiodicity). This assumption seems quite restrictive, but will be necessary in Section \ref{statement_result}. If this condition on $K$ were not fulfilled, we would simply have to reduce the state space of the matrices defined in Section \ref{results}  to $\{n\geq1 | K(n) > 0\}^q$ and to assume that $K$ is aperiodic.

\subsection{Finite range correlations}
Let $(\varepsilon_n)_{n\in\mathbb{Z}}$ be a collection of independent standard gaussian random variables (independent from $\tau$), $q\geq1$ a fixed integer, and $(a_0,\ldots,a_{q})\in \mathbb{R}^{q+1}$ such that $a_0^2+\ldots+a_{q}^2=1$. Define the disorder sequence $\omega = (\omega_n)_{n\geq0}$ by the $q$-order moving average $\omega_n = a_0 \varepsilon_n + \ldots + a_q \varepsilon_{n-q}$. Then $\omega$ is a stationary centered gaussian process and its covariance function $\rho_n := \cov(\omega_0,\omega_n)$ satisfies $\rho_0 = 1$ and $n>q \Rightarrow \rho_n = 0$. From now, the notations $\mathbb{P}$ and $\mathbb{E}$ will be associated to disorder.

\subsection{The quenched and annealed polymer measures}
We define the (constraint) quenched polymer measures, which depend on two parameters, the averaged pinning reward $h\in\mathbb{R}$ and the amplitude of disorder $\beta\geq0$:

\begin{equation}\label{def1}
\frac{dP_{N,\beta,h,\omega}}{dP} = \frac{1}{Z_{N,\beta,h,\omega}} \exp\left( \sum_{n=1}^N (\beta \omega_n + h) \delta_n \right) \delta_N
\end{equation}
where
\begin{equation}\label{def2}
Z_{N,\beta,h,\omega} = E\left( \exp\left( \sum_{n=1}^N (\beta \omega_n + h) \delta_n \right)\delta_N \right)
\end{equation}
is the partition function. We also define its annealed counterpart:
\begin{equation}\label{def3}
\frac{d(P\otimes \mathbb{P})_{N,\beta,h}}{d(P\otimes \mathbb{P})} = \frac{1}{Z^a_{N,\beta,h}} \exp\left( \sum_{n=1}^N (\beta \omega_n + h) \delta_n \right) \delta_N
\end{equation}
where
\[Z^a_{N,\beta,h} = \mathbb{E}Z_{N,\beta,h,\omega}. \]

\section{Generalities}\label{generalities}

\subsection{Free energy, phase diagram, and critical curve}

We give some results which are well-known for i.i.d. disorder, and which can be generalized to ergodic disorder (see \cite[Thm 4.6, p.96]{GG_Book}).

\begin{pr}\label{free_energy}
For all $h\in\mathbb{R}$ and all $\beta \geq 0$, there exists a nonnegative constant $F(\beta,h)$ such that, \[F(\beta,h) = \lim_{N\rightarrow +\infty} \frac{1}{N} \log Z_{N,\beta,h,\omega}\] $\mathbb{P}$-almost surely and in $L^1(\mathbb{P})$.
\end{pr}

\begin{proof}
We use the Markov property as in \cite[Prop 4.2, p.91]{GG_Book} or \cite[(3.1), p.12]{GG_Overview} to write \[\log Z^{\texttt{c}}_{N+M,\beta,h,\omega}\geq \log Z^{\texttt{c}}_{M,\beta,h,\omega} + \log Z^{\texttt{c}}_{N,\beta,h,\theta^M\omega}\]where $\theta$ is the shift operator. We then use Kingman's subadditive theorem (see \cite{Steele}). In our case, $\omega$ is ergodic because $\rho_n \stackrel{n\rightarrow \infty}{\longrightarrow} 0$ (see \cite[Chp 14, \S.2, Thm 2]{CFG}).
\end{proof}

The phase diagram $\mathbb{R}_{+}\times\mathbb{R}$ is then divided into a localized phase
\[\mathcal{L} = \left\{ (\beta,h) | F(\beta,h) > 0 \right\}\]
and a delocalized one
\[\mathcal{D} = \left\{ (\beta,h) | F(\beta,h) = 0 \right\}.\]
For all $\beta$, define the critical point $h_c(\beta):=\sup\{h\in\mathbb{R}|F(\beta,h)=0\}$. By convexity of $F$ (as the limit of convex functions), $\mathcal{D}$ is convex so the critical curve $\beta \mapsto h_c(\beta)$ is concave. Moreover, it is nonincreasing and $h_c(0)= 0$. For detailed arguments, we refer to \cite{GG_Book}.

\subsection{Annealed free energy and annealed bound}

The first difference that occurs when dealing with correlated disorder is that integrating on $\omega$ the Boltzmann factor does not yield a classical homogeneous model (see (\ref{variance}) below). As a consequence, we will need an additional argument to define the annealed free energy.

\begin{lem}[Hammersley's approximate subadditivity \cite{Hammersley}]\label{Hammersley}
Let $h : \mathbb{N} \mapsto \mathbb{R}$ be such that for all $n$, $m\geq1$,
\[h(n+m) \leq h(n) + h(m) + \Delta(n+m),\]
with $\Delta$ a non decreasing sequence satisfying:
\[\sum_{r=1}^{\infty} \frac{\Delta(r)}{r(r+1)} < \infty.\]
Then, $\displaystyle\lim_{n\rightarrow +\infty} \frac{h(n)}{n}$ exists and is finite.
\end{lem}

\begin{theo}\label{annealed_f}
For all $h\in\mathbb{R}$ and all $\beta \geq 0$, there exists a nonnegative constant $F^a(\beta,h)$ such that, \[F^a(\beta,h) = \lim_{N\rightarrow+\infty} \frac{1}{N} \log Z_{N,\beta,h}^a.\] Moreover, if $h_c^a(\beta) := \sup\{h\in\mathbb{R}|F^a(\beta,h)=0\}$ then \begin{equation}\label{AB}h_c(\beta) \geq h_c^a(\beta).\end{equation}
\end{theo}

\begin{proof}
First, we compute the variance (with respect to $\omega$) of $\sum_{n=1}^N \omega_n \delta_n$. For every realization of $\tau$, we have:
\begin{equation}\label{variance}
\var\left(\sum_{n=1}^N \omega_n \delta_n\right) = \sum_{i, j =1}^N \cov(\omega_i,\omega_j)\delta_i \delta_j
= \sum_{n=1}^N \delta_n + 2 \sum_{1\leq i < j \leq N} \rho_{j-i} \delta_i \delta_j.
\end{equation}
Then, $$Z^a_{N,\beta,h} = E\left( \exp\left( (h+\frac{\beta^2}{2})\sum_{n=1}^N \delta_n + \beta^2 \sum_{i=1}^{N-1} \sum_{k=1}^{N-i} \rho_k \delta_i \delta_{i+k}\right) \delta_N \right).$$ Now, we want some sort of superadditivity for the annealed partition function. For a polymer of size $N+M$, observe that
\begin{align*}
\sum_{1\leq i < j \leq N+M} \rho_{j-i}\delta_i \delta_j = \sum_{1\leq i < j \leq N} \rho_{j-i}\delta_i \delta_j &+ \sum_{N+1\leq i < j \leq N+M} \rho_{j-i}\delta_i \delta_j\\
&+ \sum_{1\leq i \leq N < j \leq N+M} \rho_{j-i}\delta_i \delta_j.
\end{align*}
Conditioned on the event $\{N\in\tau\}$, the second term has the same law as $\sum_{1\leq i < j \leq M} \rho_{j-i}\delta_i \delta_j$. Moreover, the third term is greater than a constant $C$ only depending only $\rho$ and $q$. We can then write
\begin{align*}
& Z^a_{N+M,\beta,h} \\&\geq E\left( \exp\left( (h+\frac{\beta^2}{2})\sum_{n=1}^{N+M} \delta_n + \beta^2 \sum_{1\leq i < j \leq N+M} \rho_{j-i}\delta_i \delta_j \right) \delta_N \delta_{N+M} \right)\\
&\geq e^{C\beta^2} Z^a_{N,\beta,h} Z^a_{M,\beta,h}
\end{align*}
and we conclude by using Lemma \ref{Hammersley} to $-\log Z^a_{N,\beta,h}$ with $\Delta(n) = - C\beta^2$. As in \cite[Prop 5.1]{GG_Book} we use Jensen's inequality to prove that \begin{equation}\label{ann_bound}F(\beta,h)\leq F^a(\beta,h),\end{equation} which in turn yields the annealed bound (\ref{AB}).
\end{proof}

When disorder is i.i.d, ($\ref{AB}$) becomes $h_c(\beta)\geq h_c(0) - \beta^2/2 := h_c^a(\beta)$ and the question of knowing whether this is an equality was studied in several papers and monographs (for example, \cite{GG_Overview}, \cite{Toninelli_Survey}, \cite{GG_Book} and references therein) where we learn that the answer depends on the values of $\alpha$ and $\beta$.

In the next subsection, the effect of correlations on $h_c^a$ will be studied.

\section{The annealed critical curve}\label{results}

\subsection{The result for $q=1$ and the reason why the technique used does not apply to $q>1$}

\begin{pr}\label{q_equals_1}
If $q=1$ then we have \[h_c^a(\beta) = h_c(0) - \frac{\beta^2}{2} - \log\left( 1 + K(1)\left( e^{\rho_1 \beta^2}-1 \right)\right)\]
\end{pr}

\begin{proof} If $q=1$, equality (\ref{variance}) gives: \[Z_{N,\beta,h}^a = E\left( \exp\left( (h+\frac{\beta^2}{2})\imath_N + \rho_1\beta^2 \sum_{n=1}^{N-1} \delta_n \delta_{n+1}\right)\delta_N \right).\] The energetic contribution of a jump can only take two values: $h+(2\rho_1+1)\beta^2/2 $ if the jump has size 1 and $h+\beta^2/2$ otherwise. The rest of the proof is a slight modification of the proof of \cite[Prop 1.1]{GG_Book}, except we must consider $K_{(q=1)}$ with $K_{(q=1)}(1) := e^{\rho_1 \beta^2}K(1)$ and $K_{(q=1)}(n) := K(n)$ if $n>1$.
\end{proof}

If $q\geq2$, the situation is more complicated because in this case we must consider the energetic contribution of a q-tuple of jumps instead of one of a single jump. For example, if $q=2$, the energetic contribution of a jump of size 1 can be $h+(1+2\rho_1)\beta^2/2$ or $h+(1+2\rho_1+ 2\rho_2)\beta^2/2$, depending on the value of the jump just before. This idea of looking at the sequence of q-tuples of consecutive inter-arrival times is developed in the next section.

\subsection{An auxiliary Markov chain and the transfer matrix}\label{statement_result}

From now we assume $q\geq2$. We will denote by $\overline{t} = (t_1,\ldots,t_q)$ a $q$-tuple in $(\mathbb{N}^*)^q$ and if $(t_n)_{n\geq1}$ is a sequence of integers, then $\overline{t}_n := (t_n,\ldots,t_{n+q-1})$. The projection on the first coordinate $\overline{t}\mapsto t_1$ will be denoted by $\pi_1$. Let $G$ be a function defined on such $q$-tuples by $G(\overline{t}) = \sum_{k=1}^q \rho_{t_1+\ldots+t_k} $, and which should be interpreted like this: if $\overline{t}$ is the q-tuple of the inter-arrival times of $q+1$ consecutive renewal points on the interface, then $G(\overline{t})$ gives the total contribution of correlations between disorder at theses points.

Notice that when we compute the value of $G$ for some $q$-tuple of inter-arrival times, any inter arrival time strictly greater than $q$ "does not count". To put it more precisely, we can consider a "cemetery state", denoted by $\star$, and define for all $t\in\mathbb{N}^*$ and $\overline{t}\in(\mathbb{N}^*)^q$, $t^* := t \mathbf{1}_{\{t\leq q\}} + \star \mathbf{1}_{\{t > q\}}$ and $\overline{t}^* = (t_1^*,\ldots,t_q^*)$. Then $G$ can be considered as a function of $\overline{t}^*$ instead of $\overline{t}$, if we adopt the following natural conventions: $\rho_{\star} = 0$ and for all $t \in \{1,\ldots,q,\star\}$, $\star + t = t + \star = \star$. From now we will use the following notations: $E = \{1,\ldots,q,\star\}$ and $K(\star) = \sum_{n>q} K(n)$.

In the following we will write $\overline{s} \rightsquigarrow \overline{t}$ (resp. $\overline{s}^* \rightsquigarrow \overline{t}^*$) if for all $i\in\{2,\ldots,q\}$, $s_i = t_{i-1}$ (resp. $s_i^* = t_{i-1}^*$). We now make the following remark: the sequence of $q$-tuples $(\overline{T}_n)_{n\geq1}$ is a Markov chain on a countable state space, and its transition probability from state $\overline{s}=(s_1,\ldots,s_q)$ to state $\overline{t}=(t_1,\ldots,t_q)$ writes \[Q(\overline{s},\overline{t}) := K(t_q) \mathbf{1}_{\{\overline{s} \rightsquigarrow \overline{t}\}}.\] Note that $Q$ is irreducible because of the positiveness of the $K(n)$'s. We now define the nonnegative matrices $Q_{\beta}$ and $Q^*_{\beta}$,which will play the role of transfer matrices, by
\[Q_{\beta}(\overline{s},\overline{t}) = e^{\beta^2 G(\overline{t})}K(t_q) \mathbf{1}_{\{\overline{s} \rightsquigarrow \overline{t}\}}\]
and
\[Q^*_{\beta}(\overline{s}^*,\overline{t}^*) = e^{\beta^2 G(\overline{t}^*)}K(t_q^*) \mathbf{1}_{\{\overline{s}^* \rightsquigarrow \overline{t}^*\}}.\]

We will write $Q^*$ instead of $Q^*_0$. Since $Q^*_{\beta}$ is an irreducible nonnegative matrix on the finite state space $E^q$, we know by the Perron-Frobenius theorem that there exists a Perron-Frobenius eigenvalue $\lambda(\beta)$ and an associated right eigenvector $\nu_{\beta}^* = (\nu_{\beta}^*(x))_{x\in E^q}$ with positive components (see \cite{Seneta}).

\subsection{Statement of the results}

We are now ready to state our main results. The first one expresses the annealed critical curve in terms of the Perron-Frobenius eigenvalue of the transfer matrix $Q^*_{\beta}$.

\begin{theo}\label{annealed_critical_curve}
For all $\alpha\geq0$, for all $\beta\geq0$, \[h_c^a(\beta) =  - \frac{\beta^2}{2} - \log \lambda(\beta).\]
\end{theo}

It seems difficult to give a nice explicit expression of $\lambda(\beta)$, since it is the Perron-Frobenius eigenvalue of a matrix of size $(q+1)^q$. For $q=2$, we have computed
\begin{equation*}
 h_c^a(\beta)= -\frac{\beta^2}{2} - \log \phi(\beta) - \log\left( \frac{1 + \sqrt{1 - \frac{\psi(\beta)}{\phi(\beta)^2}}}{2}\right)
\end{equation*}
where 
\begin{align*}
 \phi(\beta) &= 1 + K(1)(e^{(\rho_1+ \rho_2) \beta^2}-1) + K(2)(e^{\rho_2\beta^2}-1)\\
 \psi(\beta) & = 4 K(1)(1-K(1))e^{\rho_1 \beta^2}(e^{\rho_2 \beta^2}-1) \left(1 + \frac{K(2)}{1-K(1)}(e^{\rho_2 \beta^2}-1)\right).
\end{align*}

In the general case, the asymptotic behaviour of the annealed critical curve for weak disorder can be explicited:

\begin{pr}\label{asymptotic} We have \[h_c^a(\beta) \stackrel{\beta \rightarrow 0}{\sim} -\left(1 + 2 \sum_{n=1}^q \rho_n P(n\in\tau) \right)\frac{\beta^2}{2}.\]
\end{pr}

Before going into details, we outline the proof of Theorem \ref{annealed_critical_curve}. First, we introduce in Lemma \ref{new_transition} new Markov transition kernels built from the transfer matrices and an eigenvector associated to $\lambda(\beta)$. From these we give a new law for the sequence of q-tuples of consecutive inter-arrival times, to which we associate what could be called a ``q-correlated'' renewal process. This process is in fact a particular Markov renewal process (these are processes in which the return times are not necessarily i.i.d., but driven by a Markov chain, see \cite{Asmussen} on this subject). With Lemma \ref{lemma_tilde}, we link the annealed free energy of our initial model to the homogeneous free energy of the new ``q-correlated'' renewal process. This will be the starting point of the proof of Theorem \ref{annealed_critical_curve}. Note that for positive recurrent renewal processes we give a shorter proof than in the general case.

\subsection{A "q-correlated" renewal process related to the model}

For all q-tuples $\overline{t}$, define $\nu_{\beta}(\overline{t}) = \nu_{\beta}^*(\overline{t}^*)$.

\begin{lem}\label{new_transition}
$\tilde{Q}_{\beta}(\overline{s},\overline{t}) := \frac{Q_{\beta}(\overline{s},\overline{t})\nu_{\beta}(\overline{t})}{\lambda(\beta)\nu_{\beta}(\overline{s})}$ and $\tilde{Q}^*_{\beta}(\overline{s}^*,\overline{t}^*) := \frac{Q_{\beta}^*(\overline{s}^*,\overline{t}^*)\nu_{\beta}^*(\overline{t}^*)}{\lambda(\beta)\nu_{\beta}^*(\overline{s}^*)}$ are Markov transition kernels.
\end{lem}

\begin{proof}
 For $\tilde{Q}_{\beta}^*$, the result is a direct consequence of the relation $Q_{\beta}^*\nu_{\beta}^* = \lambda(\beta)\nu_{\beta}^*$ and of the positiveness of $\lambda(\beta)$ and $\nu_{\beta}$. For  $\tilde{Q}_{\beta}$, we write for all $\overline{s} = (s_1,\ldots,s_q)$,
\begin{align*}
 \sum_{\overline{t}}Q_{\beta}(\overline{s},\overline{t})\nu_{\beta}(\overline{t})  &= \sum_{t\geq 1} e^{\beta^2 G(s_2,\ldots,s_a,t)}K(t)\nu_{\beta}(s_2,\ldots,s_q,t) \\
& = \sum_{t\geq 1} e^{\beta^2 G(s_2^*,\ldots,s_q^*,t^*)}K(t)\nu_{\beta}^*(s_2^*,\ldots,s_q^*,t^*)\\
& = \sum_{t^* \in E} e^{\beta^2 G(s_2^*,\ldots,s_q^*,t^*)}K(t^*)\nu_{\beta}^*(s_2^*,\ldots,s_q^*,t^*)\\
& = \lambda(\beta) \nu_{\beta}^*(\overline{s}^*) \\
& = \lambda(\beta) \nu_{\beta}(\overline{s}).
\end{align*}
The result follows in the same way as for $\tilde{Q}_{\beta}^*$.
\end{proof}

Since $\tilde{Q}^*_{\beta}$ is a finite irreducible transition matrix (it has the same incidence matrix as $Q^*_{\beta}$, which is irreducible), it has a unique invariant probability measure that we denote by $\mu_{\beta}^*$. If we define $\mu_{\beta}$ a measure on $(\mathbb{N}^*)^q$ by $\mu_{\beta}(\overline{t}) = \frac{K(t_0)}{K(t_0^*)}\ldots\frac{K(t_{q-1})}{K(t_{q-1}^*)}\mu_{\beta}^*(\overline{t}^*)$, then

\begin{lem}\label{mu_invariant}
$\mu_{\beta}$ is the invariant probability of $\tilde{Q}_{\beta}$.
\end{lem}

\begin{proof}
By a direct computation, $\mu_{\beta}$ is a probability. Now we prove that it is invariant. For all $\overline{t} \in (\mathbb{N}^*)^q$, we have
\begin{align*}
&\sum_{\overline{s}} \mu_{\beta}(\overline{s}) \tilde{Q}_{\beta}(\overline{s},\overline{t})\\ &= \lambda(\beta)^{-1} e^{\beta^2G(\overline{t})}\nu_{\beta}(\overline{t})K(t_q) \sum_{s\geq1} \frac{\mu_{\beta}(s,t_1,\ldots,t_{q-1})}{\nu_{\beta}(s,t_1,\ldots,t_{q-1})} \\
& = \lambda(\beta)^{-1} e^{\beta^2G(\overline{t}^*)}\nu_{\beta}^*(\overline{t}^*)K(t_q)\\& \times \sum_{s\geq1} \frac{K(s)K(t_1)\ldots K(t_{q-1})}{K(s^*)K(t_1^*)\ldots K(t_{q-1}^*)} \frac{\mu_{\beta}^*(s^*,t_1^*,\ldots,t_{q-1}^*)}{\nu_{\beta}^*(s^*,t_1^*,\ldots,t_{q-1}^*)}\\
& = \lambda(\beta)^{-1} e^{\beta^2G(\overline{t}^*)}\nu_{\beta}^*(\overline{t}^*)K(t_q^*)\frac{\mu_{\beta}(\overline{t})}{\mu_{\beta}^*(\overline{t}^*)} \sum_{s\geq1} \frac{K(s)}{K(s^*)} \frac{\mu_{\beta}^*(s^*,t_1^*,\ldots,t_{q-1}^*)}{\nu_{\beta}^*(s^*,t_1^*,\ldots,t_{q-1}^*)}\\
& =\lambda(\beta)^{-1} e^{\beta^2G(\overline{t}^*)}\nu_{\beta}^*(\overline{t}^*)K(t_q^*)\frac{\mu_{\beta}(\overline{t})}{\mu_{\beta}^*(\overline{t}^*)} \sum_{s^*\in E^q} \frac{\mu_{\beta}^*(s^*,t_1^*,\ldots,t_{q-1}^*)}{\nu_{\beta}^*(s^*,t_1^*,\ldots,t_{q-1}^*)}\\
& = \mu_{\beta}(\overline{t})
\end{align*}
where for the last equality we use the fact that $\mu_{\beta}^*$ is the invariant probability of $\tilde{Q}^*_{\beta}$.
\end{proof}

We define a new law on the interarrival times $(T_n)_{n\geq1}$, denoted by $P_{\beta}$, by the following relations:
\[P_{\beta}(T_1 = t_1, \ldots, T_q = t_q) = \prod_{k=1}^q K(t_k)\]
and for all $k\geq0$
\[P_{\beta}(T_{k+q-1} = t_{q+1} | T_{k+1} = t_1, \ldots, T_{k+q}=t_q) = \tilde{Q}_{\beta}((t_1,\ldots,t_q),(t_2,\ldots,t_{q+1})) \]
To determine $T_{k+q+1}$ conditionnally to the past, only $\overline{T}_{k+1}^*$ is relevant (and not $\overline{T}_{k+1}$) since it can be checked that
\begin{align*}
P_{\beta}(T_{k+q-1} &= t_{q+1} | T_{k+1}= t_1, \ldots, T_{k+q}=t_q) \\&= \tilde{Q}_{\beta}((t_1^*,\ldots,t_q^*),(t_2^*,\ldots,t_{q+1}^*))\times \frac{K(t_{q+1})}{K(t_{q+1}^*)}\\
&= P_{\beta}(T_{k+q-1} = t_{q+1} | T_{k+1}^* = t_1^*, \ldots, T_{k+q}^*=t_q^*).
\end{align*}
Under $P_{\beta}$, $(\tau_n)_{n\geq0}$ is then a (delayed) Markov renewal process with markov modulating chain $(\overline{T}_{k-q}^*)_{k\geq q+1}$, and semi-Markov kernel: for all $n\geq1$, $x,y\in E^q$,
\[P_{\beta}(T_{k+q+1}=n, \overline{T}_{k+2}^*=y | \overline{T}_{k+1}^*=x) = \tilde{Q}_{\beta}^*(x,y) \frac{K(n)}{K(y_q)}\mathbf{1}_{\{n^* = y_q\}}.\]

\begin{lem}\label{lemma_G}
 For all $h\in\mathbb{R}$ and all $\beta\geq0$,
\[F^a(\beta,h) = \lim_{N\rightarrow+\infty} \frac{1}{N}\log E\left( e^{(h+\frac{\beta^2}{2})\imath_N + \beta^2 \sum_{n=1}^{\imath_N} G(\overline{T}_n)} \delta_N\right).\] \end{lem}

\begin{proof}
On one hand, we have by integrating over disorder the partition function:
\begin{align*}
Z^a_{N,\beta,h} &= E\left( \exp\left( \imath_N(h+\frac{\beta^2}{2}) + \beta^2 \sum_{1\leq i < j \leq N} \rho_{j-i} \delta_i \delta_j\right)\right)\\
&= E\left( \exp\left( \imath_N(h+\frac{\beta^2}{2}) + \beta^2 \sum_{i=1}^{N-1} \sum_{k=1}^{N-i} \rho_k \delta_i \delta_{i+k}\right)\right).
\end{align*}
On the other hand, $\sum_{n=1}^{\imath_N} G(\overline{T}_n) = \sum_{i=1}^N \sum_{k=1}^q \rho_k \delta_i \delta_{i+k}$. We prove the lemma by showing that there exists a constant $C(\rho,q)$ such that \[\left| \sum_{n=1}^{\imath_N} G(\overline{T}_n) -  \sum_{i=1}^{N-1} \sum_{k=1}^{N-i} \rho_k \delta_i \delta_{i+k}\right| \leq C(\rho,q).\] Indeed,
\begin{align*}
\sum_{i=1}^N \sum_{k=1}^q \rho_k \delta_i \delta_{i+k} &= \sum_{i=1}^{N-1} \sum_{k=1}^q \rho_k \delta_i \delta_{i+k} + \sum_{k=1}^q \rho_k \delta_N \delta_{N+k} \\
&= \sum_{i=1}^{N-1} \sum_{k=1}^{N-i} \rho_k \delta_i \delta_{i+k} + \sum_{k=1}^q \rho_k \delta_N \delta_{N+k} \\& + \sum_{i=N-q+1}^{N-1} \sum_{k=N-i+1}^q \rho_k \delta_i \delta_{i+k}
\end{align*}
where the second term is bounded in absolute value by $q\times \max_{i=1\ldots q}|\rho_i|$ and the third term by $\frac{q(q+1)}{2}\times \max_{i=1\ldots q}|\rho_i|$.
\end{proof}

\begin{lem}\label{lemma_tilde}
  For all $h\in\mathbb{R}$ and all $\beta\geq0$,
\[F^a(\beta,h) = \lim_{N\rightarrow+\infty} \frac{1}{N}\log E_{\beta}\left( e^{(h+ \frac{\beta^2}{2} + \log \lambda(\beta)) \imath_N} \delta_N\right).\]
\end{lem}

\begin{proof}
By decomposing on the possible values of $\imath_N$, we have on one hand:
\begin{align*}
&E\left( e^{(h+\frac{\beta^2}{2})\imath_N + \beta^2 \sum_{n=1}^{\imath_N} G(\overline{T}_n)}\delta_N\right)\\ &= \sum_{n=1}^N e^{(h+\frac{\beta^2}{2})n}\displaystyle\sum_{\substack{\overline{t}_1,\ldots\overline{t}_n\\t_1+\ldots+t_{n}= N}} e^{\beta^2 \sum_{k=1}^n G(\overline{t}_k)} Q(\overline{t}_1,\overline{t}_2)\ldots Q(\overline{t}_{n-1},\overline{t}_n)K^{\otimes q}(\overline{t}_1)\\
& = \sum_{n=1}^N \frac{e^{(h+\frac{\beta^2}{2} + \log \lambda(\beta))n}}{\lambda(\beta)^n}\displaystyle\sum_{\substack{\overline{t}_1,\ldots\overline{t}_n\\t_1+\ldots+t_{n}= N}}  Q_{\beta}(\overline{t}_1,\overline{t}_2)\ldots Q_{\beta}(\overline{t}_{n-1},\overline{t}_n)K^{\otimes q}(\overline{t}_1)\\
& = \sum_{n=1}^N e^{(h+\frac{\beta^2}{2} + \log \lambda(\beta))n}\\&\times \displaystyle\sum_{\substack{\overline{t}_1,\ldots\overline{t}_n\\t_1+\ldots+t_{n}= N}}  \frac{\nu_{\beta}(\overline{t}_1)}{\nu_{\beta}(\overline{t}_n)} \tilde{Q}_{\beta}(\overline{t}_1,\overline{t}_2)\ldots \tilde{Q}_{\beta}(\overline{t}_{n-1},\overline{t}_n)K^{\otimes q}(\overline{t}_1)
\end{align*}
and on the other hand,
\begin{align*}&E_{\beta}\left( e^{(h+\frac{\beta^2}{2} + \log \lambda(\beta))\imath_N}\delta_N\right) \\&= \sum_{n=1}^N e^{(h+\frac{\beta^2}{2} + \log \lambda(\beta))n} \displaystyle\sum_{\substack{\overline{t}_1,\ldots\overline{t}_n\\t_1+\ldots+t_n = N}}  \tilde{Q}_{\beta}(\overline{t}_1,\overline{t}_2)\ldots \tilde{Q}_{\beta}(\overline{t}_{n-1},\overline{t}_n)K^{\otimes q}(\overline{t}_1)
\end{align*}
Since $\nu_{\beta}(\overline{t}) = \nu_{\beta}^*(\overline{t}^*)$ and $\nu^*$ is a finite vector with positive components, there exists $c$ and $C$ two positive constants such that for all $\overline{t}_1,\overline{t}_n$, $c\leq \frac{\nu_{\beta}(\overline{t}_1)}{\nu_{\beta}(\overline{t}_n)} \leq C$. We conclude by using this remark and Lemma \ref{lemma_G}.
\end{proof}

\subsection{A short proof of Theorem \ref{annealed_critical_curve} in the positive recurrent case}

In accordance with Lemma \ref{lemma_tilde}, we will work on the homogeneous pinning model of the process $\tau$ under $P_{\beta}$. In the positive recurrent case, a renewal-type lemma is obtained, which allows us to conclude.

\begin{lem}\label{renewal_tilde}
If $\alpha > 1$, or if $\alpha=1$ and $L$ is such that $\sum_{n\geq1}L(n)/n < \infty$ then $\frac{\imath_N}{N}$ tends $P_{\beta}$-almost surely and in $L^1(P_{\beta})$ to a positive constant.
\end{lem}

\begin{proof}
From Lemma \ref{mu_invariant}, under $P_{\beta}$, the sequence of $q$-tuples $(T_k,\ldots,T_{k+q-1})_{k\geq0}$ is a positive recurrent Markov chain, with invariant probability measure $\mu_{\beta}$. If the previous conditions on $\alpha$ are satisfied, $\pi_1 : \overline{t}\rightarrow t_1$ (the projection on the first coordinate) is $\mu_{\beta}$-integrable. As a consequence,
\[\frac{\tau_N}{N} = \frac{1}{N}\sum_{k=1}^{N} \pi_1(T_k,\ldots,T_{k+q-1}) \stackrel{P_{\beta}-a.s.}{\longrightarrow} c := \sum_{t_2,\ldots,t_q\geq1} t_1 \times \mu_{\beta}(t_1,\ldots,t_q) < \infty.\] We deduce from this that $\frac{\imath_N}{N} \stackrel{P_{\beta}-a.s.}{\rightarrow} \frac{1}{c} > 0$ by using the inequality $\tau_{\imath_N} \leq N \leq \tau_{\imath_N+1}$. The convergence in $L^1$ follows from the Dominated Convergence Theorem. 
\end{proof}

From Lemma \ref{lemma_tilde}, $h\leq -\frac{\beta^2}{2} -\log \lambda(\beta)$ implies that $F^a(\beta,h) = 0$. Suppose now that $h = -\frac{\beta^2}{2} -\log \lambda(\beta) + \epsilon$ with $\epsilon > 0$. By Jensen's inequality, we have that \[\frac{1}{N}\log E_{\beta}\left( e^{\epsilon \imath_N} \right) \geq \epsilon \frac{E_{\beta}(\imath_N)}{N}.\]We conclude that $F^a(\beta,h) > 0$ by using Lemma \ref{renewal_tilde} and Lemma \ref{lemma_tilde}.

\subsection{Proof of Theorem \ref{annealed_critical_curve} in the general case}

We now give a proof without any assumption on $\alpha$. The starting point is Lemma \ref{lemma_tilde} and we will actually identify the free energy of the pinning model associated to the law $P_{\beta}$. Let's fix $\epsilon > 0$. We introduce the matrices
\[\tilde{Q}_{\beta,F}(\overline{s},\overline{t})= e^{- Ft_q}\tilde{Q}_{\beta}(\overline{s},\overline{t})\]
and\[\tilde{Q}^*_{\beta,F}(\overline{s}^*,\overline{t}^*)= e^{- F\phi_F(t_q^*)}\tilde{Q}^*_{\beta}(\overline{s}^*,\overline{t}^*)\]
where $\phi_F(s^*) = s^*$ if $s^*\in \{1,\ldots,q\}$ and
\[\phi_F(\star) = -\frac{1}{F}\log\frac{\sum_{t>q} e^{-Ft}K(t)}{K(\star)}\]
i.e. $\phi_F(\star)$ verifies
\begin{equation}\label{phi_star}
e^{-F\phi_F(\star)}K(\star) = \sum_{t>q} e^{-Ft} K(t).
\end{equation}
We will denote by $\Lambda(\beta,F)$ the Perron-Frobenius eigenvalue of the irreducible matrix $\tilde{Q}^*_{\beta,F}$.

\begin{lem}\label{free_energy_caract}
 There is a unique positive real denoted by $F_{\beta}(\epsilon)$ such that
\[ \Lambda(\beta,F_{\beta}(\epsilon)) = \exp(-\epsilon).\]
 \end{lem}

\begin{proof}
 Componentwise, $\tilde{Q}^*_{\beta,F}$ is smooth and strictly decreasing with respect to $F$. Since $\Lambda(\beta,F)$ is a simple root of the characteristic equation of $\tilde{Q}^*_{\beta,F}$ (see \cite[Thm 1.1]{Seneta}), $\Lambda(\beta,F)$ is also a smooth function of $F$ by the Implicit Function Theorem. From the formula (see \cite{Seneta}) \[\Lambda(\beta,F) = \max_{\substack{v\geq0\\\sum_{E^q}v_i=1}} \min_{j: v_j>0} \frac{(\tilde{Q}^*_{\beta,F}v)_j}{v_j}\] one also obtains (see \cite[Appendix A.8]{GG_Book}) that $\Lambda(\beta,F)$ is strictly decreasing in $F$ and that $\Lambda(\beta,F) \rightarrow 0$ as $F\rightarrow \infty$. Since $\Lambda(\beta,0) = 1 > \exp(-\epsilon)$, the result follows.
\end{proof}

Let $\tilde{\nu}^*$ be a Perron-Frobenius right eigenvector of $\tilde{Q}^*_{\beta,F_{\beta}(\epsilon)}$. We define $\tilde{\nu}$ by \begin{equation}\label{nu_tilde_star}\tilde{\nu}(\overline{t}) = \tilde{\nu}^*(\overline{t}^*).\end{equation}

\begin{lem}The matrices
\begin{equation}\label{P}
P_{F}(\overline{s},\overline{t}):= \tilde{Q}_{\beta,F_{\beta}(\epsilon)}(\overline{s},\overline{t})\frac{\tilde{\nu}(\overline{t})}{\tilde{\nu}(\overline{s})}
\end{equation}
and
\begin{equation}\label{Pstar}
P_{F}^*(\overline{s}^*,\overline{t}^*):= \tilde{Q}^*_{\beta,F_{\beta}(\epsilon)}(\overline{s}^*,\overline{t}^*)\frac{\tilde{\nu}^*(\overline{t}^*)}{\tilde{\nu}^*(\overline{s}^*)}
\end{equation}
are stochastic and irreducible matrices. Furthermore, if we denote by $l^*$ the invariant probability measure of $P_F^*$, then $l$ defined by
\begin{equation}\label{l}
l(\overline{s}) := l^*(\overline{s}^*) \prod_{j=1}^q \frac{K(s_j)e^{-F_{\beta}(\epsilon)s_j}}{K(s_j^*)e^{-F_{\beta}(\epsilon)\phi_{F_{\beta}(\epsilon)}(s_j^*)}}
\end{equation}
is the invariant probability measure of $P_F$.
\end{lem}

\begin{proof}
 The proof is left to the reader. It consists in straightforward computations very similar to Lemma \ref{new_transition} and Lemma \ref{mu_invariant}. We use Lemma \ref{free_energy_caract} to prove (\ref{Pstar}),and (\ref{phi_star}), (\ref{nu_tilde_star}), (\ref{Pstar}) to prove (\ref{P}).
\end{proof}

Note that, like $Q_{\beta}$ and $\tilde{Q}_{\beta}$, $P_F$ satisfies the ``consistancy`` condition \[P_F(\overline{s},\overline{t}) \neq 0 \Leftrightarrow \overline{s}\rightsquigarrow\overline{t}.\] This allows us to define a new law $P^{(F)}$ on $\tau$, the law for which $(\overline{T}_n)_{n\geq1}$ is a Markov chain with transition kernel $P_F$ and initial distribution $l$. 

\begin{lem}\label{pinning_F}
 There exists two constants $C\geq c>0$ such that
\[c e^{F_{\beta}(\epsilon)N}P^{(F)}(N \in \tau) \leq E_{\beta}\left( \exp(\epsilon \imath_N) \delta_N\right) \leq C e^{F_{\beta}(\epsilon)N}P^{(F)}(N \in \tau).\]
\end{lem}

\begin{proof}
 Decomposing the partition function and using (\ref{P}) we get
\begin{align*}
&E_{\beta}\left( \exp(\epsilon \imath_N) \delta_N\right) \\&= \sum_{n=1}^N \sum_{\substack{\overline{t}_1,\ldots,\overline{t}_{n-q+1}\\t_1+\ldots+t_{n}=N}} e^{\epsilon n} \tilde{Q}_{\beta}(\overline{t}_1,\overline{t}_2)\ldots \tilde{Q}_{\beta}(\overline{t}_{n-q},\overline{t}_{n-q+1})K^{\otimes q}(\overline{t}_1)\\
&= e^{F_{\beta}(\epsilon)N} \sum_{n=1}^N \sum_{\substack{\overline{t}_1,\ldots,\overline{t}_{n-q+1}\\t_1+\ldots+t_{n}=N}}P_F(\overline{t}_1,\overline{t}_2)\ldots P_F(\overline{t}_{n-q},\overline{t}_{n-q+1})l(\overline{t}_1) \\&\times \left( \frac{\tilde{\nu}(\overline{t}_1)}{\tilde{\nu}(\overline{t}_{n-q+1})} \frac{K^{\otimes q}(\overline{t}_1)e^{-\tilde{F}(\epsilon)(t_1+\ldots+t_q)}}{l(\overline{t}_1)} e^{\epsilon q}\right)
\end{align*}
and, from (\ref{nu_tilde_star}), (\ref{l}) and the finiteness of $E^q$, the term in parenthesis is uniformly bounded by two positive constants $C$ and $c$.
\end{proof}

From this we deduce:

\begin{lem}\label{free_energy_tilde}
 For all $\epsilon>0$, \[\lim \frac{1}{N} E_{\beta}\left( \exp(\epsilon \imath_N) \delta_N\right) = F_{\beta}(\epsilon) > 0.\]
\end{lem}

This lemma (combined with Lemma \ref{lemma_tilde}) tells us that $$F^a(\beta,-\frac{\beta^2}{2}-\log\lambda(\beta)+\epsilon) = F_{\beta}(\epsilon).$$

\begin{proof}
 Since $P^{(F)}(N \in \tau)\leq 1$, it will be sufficient to prove that \[\liminf_{N\rightarrow \infty} P^{(F)}(N \in \tau) > 0.\]
We use an argument that has been already used in the study of Markov renewal processes arising in the study of periodic pinning (see \cite[Chp VII.4]{Asmussen}, \cite{CaGiZa07} or \cite[Chp 3]{GG_Book}). We choose arbitrarily the state $\id = (1,\ldots,1) \in (\mathbb{N}^*)^q$. Consider $(\theta_n)_{n\geq0}$ the following sequence of stopping times:
\[\theta_0 = \inf\{ n\geq 1 | \overline{T}_n = \id \}\]
\[\theta_{k+1} = \inf\{ n > \theta_k | \overline{T}_n = \id \}\]
Since $(\overline{T}_n)_{n\geq0}$ is positive recurrent under $P^{(F)}$, these stopping times are finite almost surely. If we now define the process $\tau^{\theta}$ by $\tau_n^{\theta} := \tau_{\theta_n}$ then it is clear that \[P^{(F)}(N \in \tau) \geq P^{(F)}(N \in \tau^{\theta})\]
By the strong Markov property, $\tau^{\theta}$ is a (delayed) renewal process whose inter-arrival times are on average equal to 
\begin{align*}
 m := E_{\id}^{(F)}(T_1+ \ldots + T_{\theta_0-1}) & = E_{\id}^{(F)} \sum_{n=1}^{\infty} \pi_1(\overline{T}_n) \mathbf{1}_{\{\theta_0>n\}}\\
& = \sum_{\overline{t}} \pi_1(\overline{t}) E_{\id}^{(F)} \sum_{n=1}^{\infty} \mathbf{1}_{\{\overline{T}_n = \overline{t}, \theta_0>n\}}\\
& = \sum_{\overline{t}}  \pi_1(\overline{t}) \frac{l(\overline{t})}{l(\id)}\\
&= \frac{\sum_{k\geq1}k (l \circ \pi_1^{-1})(k)}{l(\id)} < \infty.
\end{align*}
By the Renewal Theorem, we have
\[P^{(F)}(N \in \tau^{\theta}) \stackrel{N\longrightarrow\infty}{\rightarrow} 1/m > 0\] and the proof is complete.
\end{proof}

Theorem \ref{annealed_critical_curve} is now a direct consequence of Lemma \ref{lemma_tilde} combined with Lemma \ref{free_energy_tilde}.

\subsection{The weak disorder asymptotic: proof of Proposition \ref{asymptotic}}

We now give some lemmas which will be useful for the proof of Proposition \ref{asymptotic}. If $I\subset E^q$ and $x,y\in E^q$ then we will denote by $Q^{*,I}$ the matrix with entries $Q^{*,I}(x,y) = Q^*(x,y) \mathbf{1}_{\{y\in I\}}$. If $M$ is an $n$ by $n$ matrix then $\com(M)$ is the matrix of the cofactors of $M$, i.e. $\com(M)(i,j) = (-1)^{i+j} \det M_{i,j}$ where $M^{i,j}$ is the $n-1$ by $n-1$ matrix obtained by deleting the i-th line and the j-th column of $M$.

\begin{lem}\label{stat_measure}
$Q^*$ is primitive and its invariant probability measure is $K^{\otimes q}(\overline{s}^*) = K(s_1^*)\ldots K(s_q^*)$.
\end{lem}

\begin{proof}
 For all $\overline{t}^* \in E^q$,
\begin{align*}
 (K^{\otimes q}Q_0^*)(\overline{t}^*) &= \displaystyle\sum_{\overline{s}^*\in E^q} K^{\otimes q}(\overline{s}^*) Q_0^*(\overline{s}^*,\overline{t}^*)\\
&= \displaystyle\sum_{s^*\in E} K^{\otimes q}(s^*,t_1^*,\ldots,t_{q-1}^*)K(t_q^*)\\
& = \displaystyle\sum_{s^*\in E} K(s^*)K(t_1^*)\ldots K(t_{q-1}^*)K(t_q^*)\\
& = K^{\otimes q}(\overline{t}^*)
\end{align*}
so $K^{\otimes q}$ is the invariant probability measure of $Q_0^*$. Moreover,
\begin{align*}
 (Q^*)^q(\overline{s}^*,\overline{t}^*) &= P\left( T_{q+1}^*= t_1^* ,\ldots, T_{2q}^* = t_q^* | T_1^* = s_1^*,\ldots, T_q^*= s_q^*\right) \\
&= \mathbb{P}\left( T_{q+1}^*= t_1^* ,\ldots, T_{2q}^* = t_q^* \right)\\
& = K^{\otimes q}(\overline{t}^*)
\end{align*}
which is positive under the assumptions of Section \ref{def_renewal}. Since $(Q_0^*)^q > 0$, $Q_0^*$ is primitive.
\end{proof}

\begin{lem}\label{stat_measure2}
 $\Tr\left( \com(\Id-Q^*)\right) \neq 0$ and for all $x\in E^q$ 
\[\frac{\Tr\left(^t\com(\Id-Q^*)Q^{*,\{x\}}\right)}{\Tr\left(^t\com(\Id-Q^*)\right)} = K^{\otimes q}(x).\]
\end{lem}

\begin{proof}
 In this proof we will use the properties of the Perron-Frobenius eigenvalue of a primitive matrix, that one can find for example in \cite{Seneta}.

We define  for all $x \in E^q$: \[p(x^*) := \frac{\Tr\left(^t\com(\Id-Q^*)Q^{*,\{x\}}\right)}{\Tr\left(^t\com(\Id-Q^*)\right)}.\] By Lemma \ref{stat_measure}, we only need to prove that $p$ is the invariant probability measure of $Q^*$. 

Since $Q^*$ is stochastic, $1$ is clearly a right eigenvalue of $Q^*$ with associated eigenvector $\id$ (the vector with $1$ on all its components). Moreover, $Q^*$ is primitive (Lemma (\ref{stat_measure})) so the Perron-Frobenius eigenvalue exists and all we have to prove is that $\vert \lambda \vert \leq 1$ for every (possibly complex) eigenvalue of $Q^*$. Indeed, if $v$ is an eigenvector associated with such an eigenvalue, and $x\in E^q$ is such that $v(x) = \displaystyle\max_{y\in E^q} \vert v(y) \vert$ then
\[\lambda v(x) = \displaystyle\sum_{y\in E^q} Q^*(x,y) v(y)\]
so $\vert \lambda \vert \vert v(x) \vert \leq \vert v(x) \vert$, i.e $\vert \lambda \vert \leq 1$. This proves that 1 is the Perron-Frobenius eigenvalue of $Q^*$, with associated eigenspace $\mathbb{R}\id$.

Now, from \cite[Ch. 1, Corollary 2]{Seneta}, we have that the rows of $^t\com(\Id-Q^*)$ are all equal to the same left eigenvector (for the eigenvalue $1$) of $Q^*$, that we will denote by $L$. A first consequence is that $\Tr\left(\com(\Id-Q^*)\right)\neq0$ because the entries of $L$ are either all positive or all negative. Another consequence is that if we define 
\[m = (m(x))_{x\in E^q} = \left(\Tr\left(^t\com(\Id-Q^*)Q^{*,\{x^*\}}\right)\right)_{x\in E^q}\]
then $m(x)=L(x)$ for all $x\in E^q$. Moreover, from the relation
\[(\Id-Q^*)^t\com(\Id-Q^*) = 0\]
we deduce that $\displaystyle\sum_{x\in E^q} m(x)= \Tr\left(^t\com(\Id-Q^*)\right)$. Since $p$ is simply $m$ renormalized by $\sum_{x \in E^q} m(x)$, it is the invariant probability of $Q^*$.
\end{proof}

\begin{proof}[Proof of Proposition \ref{asymptotic}]
In what follows, we will use the notations $Q'_0$ and $Q''_0$ as shortcuts for $\frac{\partial Q^*_{\beta}}{\partial\beta}(0)$ and $\frac{\partial^2 Q^*_{\beta}}{\partial\beta^2}(0)$.
First we will show that $\beta \mapsto \lambda(\beta)$ is infinitely differentiable ($\mathcal{C}^2$ would be enough). Let's define $\phi(\beta,\lambda) = \det(\lambda\Id -Q^*_{\beta} )$ so that $\phi(\beta,X)$ is the characteristic polynomial of $Q^*_{\beta}$, and $\phi(\beta,\lambda(\beta)) = 0$ for all $\beta$. The Perron-Frobenius eigenvalue of a nonnegative primitive matrix being a simple root of its characteristic equation, $\frac{\partial\phi}{\partial\lambda}(\beta,\lambda(\beta))\neq0$ for all $\beta\geq0$. Since $\phi$ is infinitely differentiable, the same holds for $\lambda$ by the Implicit Function Theorem.

Now, a straightforward computation shows that (we use that $\lambda(0)=1$)
\begin{eqnarray*}
\frac{\partial}{\partial\beta}\log\lambda(\beta)\lvert_{\beta=0}& = &\lambda'(0)\\
\frac{\partial^2}{\partial\beta^2}\log\lambda(\beta)\lvert_{\beta=0}& = & \lambda''(0) - \lambda'(0)^2.
\end{eqnarray*}
All we need to show then is 
\begin{eqnarray}
 \lambda'(0) &=& 0 \label{lambda_prime}\\
 \lambda''(0) &=& 2\sum_{n=1}^q \rho_n \mathbb{P}(n\in\tau). \label{lambda_seconde}
\end{eqnarray}
By derivating the relation $\phi(\beta,\lambda(\beta)) = 0$ we obtain
\begin{equation*}
 \frac{\partial\phi}{\partial\beta}(0,1) + \lambda'(0)\frac{\partial\phi}{\partial\lambda}(0,1) = 0.
\end{equation*}
We already know that $\frac{\partial\phi}{\partial\lambda}(0,1)\neq 0$ and since $Q'_0 = 0$ then $\frac{\partial\phi}{\partial\beta}(0,1)= 0$, which leads to (\ref{lambda_prime}).

All we have to do now is to prove (\ref{lambda_seconde}). A Taylor expansion of $\det(\lambda(\beta)\Id-Q^*_{\beta})$ gives:
\begin{align*}
 \det(\lambda(\beta)\Id-Q^*_{\beta}) &= \det\left(\Id - Q^* + (\lambda''(0)\Id-Q''_0)\frac{\beta^2}{2} + o(\beta^2)\right)\\
& = \Tr\left(^t\com(\Id - Q^*)(\lambda''(0)\Id-Q''_0)\right) \frac{\beta^2}{2} + o(\beta^2)
\end{align*}
where we have used the differential of the determinant: $\det(A + H) = \det(A) + \Tr(^t\com(A)H) + o(\Vert H \Vert)$. But since $\det(Q^*_{\beta}-\lambda(\beta)\Id) = 0$ we have
\begin{equation*}
 \Tr\left(^t\com(\Id - Q^*)(\lambda''(0)\Id-Q''_0)\right) = 0
\end{equation*}
which yields
\begin{equation*}
 \lambda''(0) = \frac{\Tr(^t\com(\Id - Q^*)Q''_0)}{\Tr(^t\com(\Id - Q^*))}.
\end{equation*}
Note that $\Tr(^t\com(\Id - Q^*)) \neq 0$ (Lemma \ref{stat_measure}).

Let's now consider $Q''_0$ as a function of $(\rho_n)_{1\leq n \leq q}$. Observe that \[Q''_0(\overline{s}^*,\overline{t}^*) = 2 G(\overline{t}^*) Q^*_0(\overline{s}^*,\overline{t}^*)\] so $Q''_0$ linearly depends on $(\rho_n)_{1\leq n \leq q}$. We have then
\begin{equation*}
 Q''_0(\rho_1,\ldots,\rho_q) = Q''_0(0,\ldots,0,\rho_q) + Q''_0(\rho_1,\ldots,\rho_{q-1},0)
\end{equation*}
The result of the theorem is clearly true for $q=1$ (remember that we have an explicit expression of $h_c^a(\beta)$ in this case, see Proposition \ref{q_equals_1} ) so we can suppose that it is true for a $(q-1)$-order moving average and show that the result holds for $q$. The induction hypothesis then implies
\begin{equation*}
 \frac{\Tr(^t\com(\Id - Q^*)Q''_0(\rho_1,\ldots,\rho_{q-1},0))}{\Tr(^t\com(\Id - Q^*))} = 2\sum_{n=1}^{q-1} \rho_n P(n\in\tau)
\end{equation*}
so the only thing left to prove is that
\begin{equation}\label{aux2}
 \frac{\Tr(^t\com(\Id - Q^*)Q''_0(0,\ldots,0,\rho_q))}{\Tr(^t\com(\Id - Q^*))} = 2\rho_q P(q\in\tau).
\end{equation}
Let's define $I_q = \{\overline{s}^*\in E^q \hbox{ s.t. } \rho_q \hbox{ appears in } G(\overline{s}^*)\}$ and notice that
\begin{equation*}
 Q''_0(0,\ldots,0,\rho_q) = 2 \rho_q Q^{*,I_q}.
\end{equation*}
We obtain from Lemma \ref{stat_measure2}:
\begin{align*}
 \frac{\Tr(^t\com(\Id - Q^*)Q^{*,I_q})}{\Tr(^t\com(\Id - Q^*))} &= \displaystyle\sum_{\overline{t}^*\in I_q} K^{\otimes q}(\overline{t}^*)\\ &= P(q\in\tau)
\end{align*}
which proves (\ref{aux2}).
\end{proof}

\section{The irrelevance regime}\label{irr_section}

In this section we will work with free partition functions (remove the $\delta_N$ in definitions (\ref{def1}),(\ref{def2}) and (\ref{def3})). This has no incidence on the free energy.

\subsection{Introduction and statement of the result}

The following result states that under some assumptions on $K$ and $\beta$, quenched and annealed critical curves and exponents are the same. This is the irrelevance regime.

\begin{theo}\label{result}
If $\omega$ is a gaussian moving average of finite order $q$ and if $\alpha\in(0,1/2)$ or if $\alpha=1/2$ and $L$ is such that
\begin{equation*}
\sum_{n=1}^{\infty} \frac{1}{nL(n)^2} < \infty
\end{equation*}
then there exists $\beta_0>0$ such that for $\beta\leq\beta_0$, $h_c(\beta) = h_c^a(\beta)$ and
\begin{equation}\label{lim}
\lim_{h\rightarrow h_c^a(\beta)^+} \frac{\log(F(\beta,h))}{\log(h-h_c^a(\beta))} = \frac{1}{\alpha}
\end{equation}
\end{theo}

Disorder irrelevance has been proved by several authors, with different methods, in the case of i.i.d. disorder (see \cite{lacoin_irrelevance},\cite{Fabio_replica} and \cite{Alexander_Sido}). A key element is the control of the second moment of the partition function at the annealed critical point, which is linked to the exponential moments of the number of intersections between two replicas of the initial renewal process $\tau$. As in \cite{lacoin_irrelevance}, we will establish (\ref{lim}) by proving separately the $\liminf$ and the $\limsup$ parts. The $\liminf$ part is just a consequence of Jensen's inequality $F(\beta,h) \leq F^a(\beta,h)$ and of the behaviour of $F^a(\beta,h)$ near the annealed critical point. The $\limsup$ part relies on the control of the second moment. In our case, additional difficulties arise from the presence of a Markov renewal process instead of a classical renewal process at the annealed critical point. Moreover the law of this Markov renewal process depends on $\beta$, so we will tackle a problem of continuity in $\beta$ (see end of Section \ref{Intersection_markovrenewal}). Once the second moment is controlled, we use arguments from \cite{lacoin_irrelevance} to conclude. Unlike what the title of \cite{lacoin_irrelevance} suggests, there is no martingale involved in our problem.

Henceforth, we assume $\alpha$ satisfies the assumption of Theorem \ref{result}.

\subsection{The $\liminf$ part}

The following proposition tells us that at the neighbourhood of the annealed critical point, the annealed free energy has the same behaviour as the homogeneous one.
\begin{pr}\label{ann_expo}
There exists a slowly varying function $L'$ such that
\begin{equation*}
F^a(\beta,h_c^a(\beta)+\Delta) \stackrel{\Delta\searrow 0}{\sim} L'(\Delta)\Delta^{1/\alpha}.
\end{equation*}
\end{pr}

\begin{proof}
The annealed free energy is defined by the implicit equation
\begin{equation*}
\Lambda(\beta,F^a(\beta,h_c^a(\beta)+\Delta)) = e^{-\Delta}
\end{equation*}
where $\Lambda(\beta,F)$ is the Perron-Frobenius eigenvalue of $\tilde{Q}^*_{\beta,F}$ (see Lemma \ref{free_energy_caract} and \ref{free_energy_tilde}). This can be rewritten as:
\begin{equation*}
1 - \Lambda(\beta,F^a(\beta,h_c^a(\beta)+\Delta)) = 1 - e^{-\Delta}
\end{equation*}
and since the right-hand term is of the order of $\Delta$ when $\Delta$ goes to $0$, it is enough to prove that the left-hand term is of the order of $\Delta^{\alpha}$.Indeed, if $\overline{t}^*$ is such that $t_q^*\in\{1,\ldots,q\}$ then
\begin{equation*}
\tilde{Q}^*_{\beta,F}(\overline{s}^*,\overline{t}^*) -\tilde{Q}^*_{\beta}(\overline{s}^*,\overline{t}^*) \sim_{F\searrow0} -\hbox{cste } F t_q
\end{equation*}
but if $t_q^* = \star$, we have by Abelian arguments
\begin{equation*}
\tilde{Q}^*_{\beta,F}(\overline{s}^*,\overline{t}^*) -\tilde{Q}^*_{\beta}(\overline{s}^*,\overline{t}^*) = -L_{\star}(1/F)F^{\alpha}
\end{equation*}
where $L_{\star}$ is a slowly varying function. We conclude the proof by writing
\begin{align*}
\Lambda(\beta,F) - 1 &= \Lambda(\tilde{Q}^*_{\beta,F}) - \Lambda(\tilde{Q}^*_{\beta})\\
& \sim_{F \searrow 0} D\Lambda_{\tilde{Q}^*_{\beta}}(\tilde{Q}^*_{\beta,F} - \tilde{Q}^*_{\beta})\\
& = -\hbox{cste }L_{\star}(1/F)F^{\alpha}.
\end{align*}
where $\Lambda$ is a differentiable function of the $(q+1)^{2q}$ entries of positive matrices.
\end{proof}

\subsection{The $\limsup$ part}

We adopt the following notations:
\begin{align*}
K_{\beta,x,y}(n) = P_{\beta}(T_k = n, \overline{T}^*_{k-q+1} = y | \overline{T}^*_{k-q} = x)
\end{align*}
and
\begin{equation*}
P_{\beta,x,y}(n\in\tau) = \sum_{k\geq0} P_{\beta}(\tau_k = n, \overline{T}^*_{k-q+1} = y | \overline{T}^*_{1-q} = x)
\end{equation*}
This section is organized as follows: in a first part we look at the intersection between two replicas of a Markov renewal process under the law $P_{\beta}$. From this we control in a second part the second moment of the partition function at the annealed critical point. In a last part, we exploit this result to obtain the $\limsup$ part of Theorem \ref{result}.

\subsubsection{Intersection of Markov renewal processes}\label{Intersection_markovrenewal}

The main result of this part is:

\begin{pr}\label{intersection}
There exists $\beta_0>0$ such that for all $\beta \leq \beta_0$ and for all $l\in\{0,\ldots,q\}$,
\[E_{\beta}^{\otimes 2}\left( \exp(\beta^2 \sum_{n\geq 1} \delta_n^{(1)}\delta_{n+l}^{(2)}) \right) < \infty.\]
\end{pr}

As it will be explained further in the proof, it is enough to focus on the case $l=0$, when the term inside the exponential is the number of intersections of two independent copies of a Markov renewal process with law $P_{\beta}$. We begin with the following observation:

\begin{pr}
If $\tau^{(1)}$ and $\tau^{(2)}$ are two independent copies of a Markov renewal process with law $P_{\beta}$, then $\tau^{(1)}\cap\tau^{(2)}$ is a (delayed) Markov renewal process.
\end{pr}

The proof is left to the reader. It is a matter of writing that conditionally on the event that $\tau^{(1)}$ and $\tau^{(2)}$ meet at some point $n$, then the future, in particular the next intersection point, only depends on the states of the Markov modulating chains of $\tau^{(1)}$ and $\tau^{(2)}$ at $n$.

In the above proposition, the term Markov renewal process has to be understood in the large sense: it can happen (and actually it will be the case in the range of $\alpha$'s we consider) that $(\tau^{(1)}\cap\tau^{(2)})_n = +\infty$ for some $n\geq1$. We will denote by $P^{\cap}_{\beta}$ the law of this intersection Markov renewal process, with Markov modulating chain in $(E^q) ^2$, and $(K^{\cap}_{\beta,x,y}(n))_{n\geq 1, x,y\in (E^q)^2 }$ its semi-Markov (sub)kernel. Hence we have to prove that $E_{\beta}^{\cap}\left(\exp(\beta^2 \sum_{n\geq 1} \delta_n) \right) < \infty$ if $\beta$ is small enough.

We define the following matrices of Laplace transforms (for $\lambda\geq 0$):
\begin{align*}
\varphi_{\beta,x,y}(\lambda) &:= \sum_{n\geq1} e^{-\lambda n} P_{\beta,x,y}^{\cap}(n\in\tau)\\
\phi_{\beta,x,y}(\lambda) &:= \sum_{n\geq1} e^{-\lambda n} K_{\beta,x,y}^{\cap}(n).
\end{align*} 

Notice that $\phi_{\beta}(0)$ is the matrix of the $P_{\beta,x,y}^{\cap}(\tau_1<\infty)$'s for $x,y\in (E^q)^2$.

\begin{pr}\label{theta_beta}
The matrix $\phi_{\beta}(0)$ is irreducible and nonnegative. If we denote by $\theta(\beta)$
 its Perron-Frobenius eigenvalue then
\begin{enumerate}
\item $\theta(0) = P^{\otimes 2}((\tau^{(1)}\cap\tau^{(2)})_1<\infty) = 1 - (\sum_{n\geq 0} P(n\in\tau)^2)^{-1} < 1$.
\item For all $\beta$, there exists a constant $c$ such that
\[P_{\beta}^{\cap}\left(\sum_{n\geq 1}\delta_n \geq N\right) \leq c\times \theta(\beta)^N.\]
\end{enumerate} 
\end{pr}

\begin{proof}
 First we prove the irreducibility. Let $x=(x^{(1)},x^{(2)})$ and $y=(y^{(1)},y^{(2)})$ be in $(E^q)^2$. We want to prove that there exists a sequence $x_0:=x, x_1, x_2, \ldots, x_i = y$ with $i\geq1$ such that \[\prod_{k=0}^i P_{\beta,x_k,y_k}^{\cap}(\tau_1<\infty) >0.\] It is enough to show that
 \begin{equation}\label{cdt_irr}
  \prod_{k=0}^i K_{\beta,x_k,y_k}^{\cap}(n_k) >0
 \end{equation} for some $n_k\geq1$.
 One can find without much difficulty a path of positive probability on which $\tau^{(1)}$ starts from $x^{(1)}$, $\tau^{(2)}$ starts from $x^{(2)}$ and they intersect at some point where respectively they are in states $y^{(1)}$ and $y^{(2)}$. This path provides suitable $i$ and $(x_k, n_k)_{1\leq k\leq i}$.

For the first point, we will only prove the first part of the equality, that is $\theta(0) = P^{\otimes 2}((\tau^{(1)}\cap\tau^{(2)})_1<\infty)$. The other part has been stated several times in the literature (see \cite{Fabio_replica} for instance). Remember that at $\beta=0$ the Markov renewal process is in fact the initial (classical) renewal process, and so the quantities $P_{x,y}^{\cap}(n\in\tau)$, $K_{x,y}^{\cap}(n)$ and $P_{x,y}^{\cap}(\tau_1<\infty)$ do not depend on $x$. As a consequence, the quantity \[P^{\otimes 2}((\tau^{(1)}\cap\tau^{(2)})_1<\infty) = \sum_{y\in (E^q)^2} P_{x,y}^{\cap}(\tau_1<\infty)\]
 is an eigenvalue of $\phi_{0}(0)$ with positive right eigenvector $\id$.

For the last point we have
 \begin{align*}
  P_{\beta}^{\cap}\left(\sum_{n\geq_1}\delta_n \geq N\right) &\leq P_{\beta}^{\cap}(\tau_1<\infty,\ldots, \tau_N<\infty)\\
  &\leq \sum_{x_0,\ldots,x_N\in(E^q)^2} \prod_{i=0}^{N-1} P_{\beta,x_i,x_{i+1}}^{\cap}(\tau_1<\infty)\\
  & \leq \sum_{x_0,x_N} \left( \phi_{\beta}(0)^N\right)_{x_0,x_N}\\
  & \leq c\times \theta(\beta)^N.
  \end{align*}
\end{proof}

This proposition implies that if $\beta$ is such that $e^{\beta^2}\theta(\beta)<1$ then
\[E_{\beta}^{\cap}\left( \exp(\beta^2\sum_{n\geq1}\delta_n)\right) < \infty.\]

In other words, the only thing left to prove is that for $\beta$ small enough, $e^{\beta^2}\theta(\beta)<1$. Actually we will prove that $\theta(\beta)$ is continuous at $\beta=0$. Since we do not have direct access to $P_{\beta}^{\cap}$, we first find a formula which is analogous to
\[P^{\otimes 2}((\tau^{(1)}\cap\tau^{(2)})_1<\infty) = 1 - (\sum_{n\geq0} P(n\in\tau)^2)^{-1},\]
i.e. which relates $\theta(\beta)$ to sums of Green functions of $P_{\beta}^{\cap}$.

\begin{pr}\label{varphi_finite}
The matrix $\varphi_{\beta}(0)$ has finite components and is irreducible. If we denote by $\vartheta(\beta)$ the Perron-Frobenius eigenvalue of $\varphi_{\beta}(0)$ then \[\theta(\beta) = 1 - \vartheta(\beta)^{-1}.\]
\end{pr}

Before proving Proposition \ref{varphi_finite}, we need a lemma, and for this lemma we need additional notations. We define $E^q_{\star} = \{x\in E^q : x_q = \star\}$ and for all $x,y\in E^q$,
\[\hat{K}_{x,y}(n) = \frac{K_{\beta,x,y}(n)}{\tilde{Q}_{\beta}^*(x,y)}\]
which is the probability under $P_{\beta}$ of making a jump of size $n$ knowing departure state $x$ and arrival  state $y$. The letter $\beta$ is omitted because the $\hat{K}_{x,y}(n)$'s do not depend on it: actually, if $y\in E^q_{\star}$, then $\hat{K}_{x,y}(n) = \frac{K(n)}{K(\star)}\mathbf{1}_{n>q}$; otherwise, $\hat{K}_{x,y}(n) = \mathbf{1}_{n=y_q}$.

\begin{lem}\label{convolution}
For all $x_0, x_1, \ldots, x_k \in E^q$,
\begin{equation}\label{c1}
(\hat{K}_{x_0,x_1}*\hat{K}_{x_1,x_2}*\ldots *\hat{K}_{x_{k-1},x_k})(n) \stackrel{n\rightarrow\infty}{\sim} \frac{L(n)}{K(\star) n^{1+\alpha}} \times |\{1\leq l \leq k, x_l \in E^q_{\star}\}|
\end{equation}
and there exists $c>0$ such that for all $k,x_0, x_1, \ldots, x_k \in E^q$ and $n>kq$
\begin{equation}\label{c2}
(\hat{K}_{x_0,x_1}*\hat{K}_{x_1,x_2}*\ldots *\hat{K}_{x_{k-1},x_k})(n) \leq k^c \frac{L(n)}{K(\star) n^{1+\alpha}}
\end{equation}
\end{lem}

\begin{proof}
Assertion (\ref{c1}) comes from the fact that if $q(n)=\tilde{L}(n)/n^{1+\alpha}$ is a probability kernel with $\tilde{L}$ a slowly varying function, then $q^{*k}(n) \sim kq(n)$ (see \cite[Lemma A.5]{GG_Book}) and only the $\hat{K}_{x,y}$'s for which $y\in E^q_{\star}$ contribute to the tail behaviour. For $k=1$, (\ref{c2}) is clearly true. One can adapt the induction in the proof of \cite[Lemma A.5]{GG_Book} to conclude.
\end{proof}

\begin{proof}[Proof of Proposition \ref{varphi_finite}]
First we prove finiteness of the components. Let $x = (x_1, x_2), y = (y_1, y_2)$ be in $(E^q)^2$. We have 
\begin{equation}\label{series}
\sum_{n\geq 0} P_{\beta,x,y}^{\cap}(n\in\tau) = \sum_{n\geq 0} P_{\beta,x_1,y_1}(n\in\tau) P_{\beta,x_2,y_2}(n\in\tau)
\end{equation}
so we have to look at the tail behaviour of the $P_{\beta,x_1,y_1}(n\in\tau)$'s. But for all $x,y\in E^q$ we can write
\begin{equation*}
 P_{\beta,x,y}(n\in\tau) = P_{\beta,x}(n\in\tau_{y})
\end{equation*}
where $\tau_{y,n}=\tau_{\theta_n(y)}$, and 
\begin{align*}
 \theta_0(y) &= \inf\{k\geq 0 , \overline{T}^*_{k-q+1} = y\}\\
\theta_{n+1}(y) &= \inf\{k > \theta_n(y), \overline{T}^*_{k-q+1} = y\}
\end{align*}
By the markov renewal property, under $P_{\beta}$, for all $y\in E^q$, $\tau_{y}$ is a (delayed, because we can start at $x\neq y$) classical renewal process. We are then left with proving that the interarrival distribution of $\tau_y$ has (approximately) the same tail behaviour as the original renewal process (which satisfies the assumptions of Theorem \ref{result}). We then conclude with the result of \cite{Doney} on renewal theorems with infinite mean to show that the series in (\ref{series}) converges. We fix the state $y\in E^q$, and write $\theta = \theta_0(y)$, which is finite almost surely ($(\overline{T}^*_n)_{n\geq1}$ is a recurrent Markov chain). We write $J_n = \overline{T}^*_{n-q}$ the markov modulating chain. Then
\begin{align*}
 P_{\beta,y}(T_1 + \ldots + T_{\theta} = n) = \sum_{k\geq1} E_{\beta,y}\left( \mathbf{1}_{\{\theta=k\}}P_z(T_0 + \ldots + T_{\theta} = n| J_0 \ldots J_k)) \right)
\end{align*}
From our previous remark on the laws $\hat{K}_{y,z}$ and Lemma \ref{convolution} we have
\begin{align*}
 P_{\beta,y}(T_0 + \ldots + T_{\theta} = n| J_0 \ldots J_k) &= (K_{J_0,J_1}*K_{J_1,J_2}*\ldots *K_{J_{k-1},J_k})(n)\\
& \stackrel{n\rightarrow\infty}{\sim} \frac{L(n)}{n^{1+\alpha}}\times \frac{N_k^{\star}}{K(\star)}
\end{align*}
where
\begin{equation*}
 N_k^{\star} := \sum_{i=1}^k \mathbf{1}_{\{J_i \in E^q_{\star}\}}.
\end{equation*}
From Markov chain theory (see \cite[Chap I.3]{Asmussen} for example),
\begin{equation*}
 E_{\beta,y} \sum_{k\geq1} \mathbf{1}_{\{\theta=k\}} N_k^{\star} = E_{\beta,y} N_{\theta}^{\star} = \frac{l_{\beta}^*(E^q_{\star})}{l_{\beta}^*(y)}
\end{equation*}
where $l_{\beta}^*$ is the invariant probability measure of $\tilde{Q}^*_{\beta}$. Finally we have
\begin{equation}\label{tail}
 P_{\beta,y}(T_0 + \ldots + T_{\theta} = n) \sim \frac{l_{\beta}^*(E^q_{\star})}{l_{\beta}^*(y)K(\star)}\times \frac{L(n)}{n^{1+\alpha}},
\end{equation}
but one has to justify the interchange of the integration and the asymptotic equivalent. Indeed, from the upper bound (\ref{c2}) in Lemma \ref{convolution}, one can apply the Dominated Convergence Theorem, because $\mathbb{E}(\theta^c)<\infty$ (it is not hard to see that the tail of $\theta$ decays exponentially fast).

We prove the last point of the proposition. The following Markov renewal equation hold: for all $x,y\in(E^q)^2$, 
 \begin{align}\label{markov_renewal_equation}
  P_{\beta,x,y}^{\cap}(n\in\tau) &= \delta_{x,y}\mathbf{1}_{n=0} + \sum_{z\in(E^q)^2} \sum_{k=1}^n P_{\beta,x,z}^{\cap}(n-k\in\tau) K_{\beta,z,y}^{\cap}(k)\\
  & =  \delta_{x,y}\mathbf{1}_{n=0} + \sum_{z\in(E^q)^2} \sum_{k=1}^n K_{\beta,x,z}^{\cap}(k) P_{\beta,z,y}^{\cap}(n-k\in\tau) 
 \end{align}
 Taking the Laplace transforms we get for $\lambda>0$
 \begin{equation*}
  \varphi_{\beta}(\lambda) = \Id + \varphi_{\beta}(\lambda)\phi_{\beta}(\lambda) = \Id + \phi_{\beta}(\lambda)\varphi_{\beta}(\lambda).
 \end{equation*}
Thanks to the first part of the proposition, that has been just proved, we can take the limit as $\lambda$ goes to $0$, which yields
 \begin{equation*}
  \varphi_{\beta}(0)(\Id-\phi_{\beta}(0))=(\Id-\phi_{\beta}(0))\varphi_{\beta}(0)=\Id
 \end{equation*}
 from which we can conclude.
\end{proof}

As a consequence, we will prove that $\vartheta(\beta)$ is continuous at $\beta=0$, which is the same as proving that for all $x,y$ in $(E^q)^2$, the series $\sum_{n\geq0} P_{\beta,x,y}^{\cap}(n\in\tau)$ are continuous at $\beta=0$.

It is not difficult to see that for all $n\geq0$, $x,y\in(E^q)^2$, the Green function $P^{\cap}_{\beta,x,y}(n\in\tau)$ is continuous in $\beta$ but the continuity of the series $$\sum_{n\geq0} P^{\cap}_{\beta,x,y}(n\in\tau)$$ is not immediate. We will see that the last quantity can be written as the $L^2$ norm of a certain function. The continuity will thus be proved on this $L^2$ norm via the Dominated Convergence Theorem.

We define the following Fourier series:
\[\hat{\phi}_{\beta,x,y}(\theta) = \sum_{n\geq1} e^{i\theta n} K_{\beta,x,y}(n)\]
\[\hat{\varphi}_{\beta,x,y}(\theta) = \sum_{n\geq0} e^{i\theta n} P_{\beta,x,y}(n\in\tau)\]
\[\hat{\varphi}^{\texttt{sym}}_{\beta,x,y}(\theta) = \sum_{n\in\mathbb{Z}} e^{i\theta n} P_{\beta,x,y}(|n|\in\tau)\]
The matrix $\hat{\phi}_0(\theta)$ will be written $\hat{\phi}(\theta)$. The functions $\hat{\phi}_{\beta,x,y}$ are continuous whereas $\hat{\varphi}_{\beta,x,y}(\theta)$ and $\hat{\varphi}^{\texttt{sym}}_{\beta,x,y}(\theta)$ are in $L^2(-\pi,\pi)$ (the space of functions which are square integrable with respect to the Lebesgue measure on $(-\pi,\pi)$), because of our knowledge on the decay of the $P_{\beta,x,y}(n\in\tau)$'s (cf previous section).

\begin{pr}\label{Fourier}
The matrix $\Id - \hat{\phi}_{\beta}(\theta)$ is $\theta$-almost everywhere invertible, and
\[P_{\beta,x,y}(n\in\tau) = \frac{1}{2\pi}\int_{-\pi}^{\pi} e^{-in\theta} \left(2\re((\Id - \hat{\phi}_{\beta}(\theta))^{-1}_{x,y}) -1\right) d\theta.\]
Furthermore there exists a positive constant $C$ such that for $\beta$ small enough, for all $x,y\in E^q$
\begin{equation}\label{domination}
|[(\Id - \hat{\phi}_{\beta}(\theta))^{-1}]_{x,y}| \leq C \sup_{s,t\in E^q} |[(\Id - \hat{\phi}(\theta))^{-1}]_{s,t}|
\end{equation}
$\theta$- almost surely.
\end{pr}

\begin{proof}
Let \[\hat{\phi}_{\beta}^{(x)}(\theta) := \sum_{n\geq1} e^{i n\theta} P_{\beta}(\tau_{x,k+1}-\tau_{x,k}=n)\] be the characteristic function of the interarrival times of $\tau_x$ under $P_{\beta}$ (cf previous section) and
\[\hat{\phi}_{\beta}^{(x,y)}(\theta) := \sum_{n\geq1} e^{i n\theta} P_{\beta,x}(\tau_{\theta_0(y)}=n),\] which are continuous functions (the term in the sum is bounded in absolute value by a probability). From our aperiodicity conditions, $\hat{\phi}_{\beta}^{(x)}(\theta)=1$ if and only if $\theta=0$. From \cite[Lemma 3.1.1]{GarsiaLamperti} (the proof can be found in \cite[Chap. II.9]{Spitzer_PRW}) we have $\theta$-almost everywhere
\[\hat{\varphi}^{\texttt{sym}}_{\beta,y,y}(\theta) = \frac{1}{1 - \hat{\phi}_{\beta}^{(y)}(\theta)}\]
and more generally, for all $x,y\in E^q$, one can show by decomposing the probability (under $P_{\beta}$) starting from $x$ that $n$ is in $\tau_y$ according to the value of $\tau$ the first time state $y$ is reached, and by taking the Fourier transform, that
\[\hat{\varphi}^{\texttt{sym}}_{\beta,x,y}(\theta) = 1 + \frac{\hat{\phi}_{\beta}^{(x,y)}(\theta)}{1 - \hat{\phi}_{\beta}^{(y)}(\theta)}\]
From the Markov renewal equations (\ref{markov_renewal_equation}), written for $P_{\beta}$ instead of $P_{\beta}^{\cap}$, one can deduce that almost everywhere,
\[\hat{\varphi}_{\beta}(\theta) = \Id + \hat{\varphi}_{\beta}(\theta)\hat{\phi}_{\beta}(\theta) = \Id + \hat{\phi}_{\beta}(\theta)\hat{\varphi}_{\beta}(\theta),\]
which proves the first part of the proposition. Then we can write
\begin{align*}
P_{\beta,x,y}(n\in\tau) & = \frac{1}{2\pi}\int_{-\pi}^{\pi} e^{-in\theta} \hat{\varphi}^{\texttt{sym}}_{\beta,x,y}(\theta)d\theta \\  & = \frac{1}{2\pi}\int_{-\pi}^{\pi} e^{-in\theta} \left( 2 \re(\hat{\varphi}_{\beta,x,y}(\theta)) -1 \right) d\theta\\ & = \frac{1}{2\pi}\int_{-\pi}^{\pi} e^{-in\theta} \left( 2 \re([\Id - \hat{\phi}_{\beta}(\theta)]^{-1}_{x,y}) -1 \right) d\theta,
\end{align*}
which ends the proof of the first part.

Now we prove the second part of the proposition (equation (\ref{domination})). We recall that
\[\hat{\phi}_{\beta,x,y}(\theta) = \hat{\phi}_{x,y}(\theta)\frac{e^{\beta^2 G(y)}\nu^*_{\beta}(y)}{\lambda(\beta)\nu^*_{\beta}(x)}\]
and $\hat{\phi}_{\beta}:=\hat{\phi}_{\beta}(\theta=0)$ is simply the matrix $Q^*$, so \begin{equation}\label{vp}\lambda(\beta) = \sum_{y\in E^q} e^{\beta^2 G(y)} \frac{\nu^*_{\beta}(y)}{\nu^*_{\beta}(x)}\hat{\phi}_{x,y}.\end{equation}

We define the matrix $R_{\beta}(\theta) = \lambda(\beta)(\Id - \hat{\phi}_{\beta}(\theta)) - (\Id - \hat{\phi}(\theta))$. It is enough to prove that the coefficients of $R_{\beta}(\theta)(\Id - \hat{\phi}(\theta))^{-1}$ decrease to $0$ as $\beta$ tends to $0$, uniformly in $\theta$. This is not immediate because there is a singularity at $\theta = 0$. Recall that 
\[ (\Id - \hat{\phi}_{\beta}(\theta))^{-1} = \frac{^t\com(\Id - \hat{\phi}(\theta))}{\det(\Id - \hat{\phi}(\theta))}.\]
We know that as $\theta$ goes to $0$, $\det(\Id - \hat{\phi}(\theta))$ is of the lowest order among the $(\hat{\phi}_{x,y}-\hat{\phi}_{x,y}(\theta))$'s, for $x,y\in E^q$, so we have to look at the elements of $R_{\beta}(\theta)^t\com(\Id - \hat{\phi}(\theta))$. More precisely we have to check that
\begin{itemize}
\item as $\theta$ goes to $0$ there are no terms of order $0$,
\item all terms of higher order have coefficients which go to $0$ as $\beta$ goes to $0$.
\end{itemize}
On one hand we have
\[(\Id - \hat{\phi}(\theta))_{x,x} = (\hat{\phi}_{x,x}- \hat{\phi}_{x,x}(\theta)) + \sum_{y\neq x} \hat{\phi}_{x,y}\]
and for $x\neq y$,
\[(\Id - \hat{\phi}(\theta))_{x,y} = (\hat{\phi}_{x,y} - \hat{\phi}_{x,y}(\theta)) + \hat{\phi}_{x,y}\]
On the other hand, we easily compute, using (\ref{vp}),
\[R_{\beta}(\theta)_{x,x} = \epsilon_{\beta}(x,x)(\hat{\phi}_{x,x}- \hat{\phi}_{x,x}(\theta)) + \sum_{y\neq x} \epsilon_{\beta}(x,y) \hat{\phi}_{x,y}\]
\[R_{\beta}(\theta)_{x,y} = \epsilon_{\beta}(x,y)(\hat{\phi}_{x,y} - \hat{\phi}_{x,y}(\theta)) - \epsilon_{\beta}(x,y)\hat{\phi}_{x,y}\]
where $\epsilon_{\beta}(x,y) := \left( e^{\beta^2 G(y)}\frac{\nu^{\star}_{\beta}(y)}{\nu^{\star}_{\beta}(x)} - 1\right)$ (which tends to $0$ as $\beta$ goes to $0$).
From these expressions, one can verify without much difficulty that the second point is satisfied. For the terms of order $0$ in $\theta$ of $R_{\beta}(\theta)^t\com(\Id - \hat{\phi}(\theta))$, it is equal to $0$; it is shown by computation, using the fact that $^t\com(\Id - \hat{\phi})$ is constant on its columns (see Section \ref{asymptotic}).
\end{proof}

We can now conclude this first part with

\begin{proof}[Proof of Proposition \ref{intersection}]. 
We begin with $l=0$, i.e. we show that $$E^{\cap}_{\beta}(\exp(\beta^2 \sum_{n\geq1}\delta_n))<\infty$$ for $\beta$ small enough. From Proposition \ref{theta_beta}, it is enough to show that $e^{\beta^2}\theta(\beta)<1$ for $\beta$ small and as $\theta(0)<1$, we will show that $\theta(\beta)\rightarrow \theta(0)$ as $\beta\rightarrow 0$. From Proposition \ref{varphi_finite}, this reduces to the continuity of $\vartheta(\beta)$ at $\beta=0$ and so, to the continuity of the series \[\sum_{n\geq0} P^{\cap}_{\beta,x,y}(n\in\tau)\] at $\beta=0$, for all $x,y\in(E^q)^2$. From (\ref{series}) one can write
\begin{equation*}
\sum_{n\geq0} P^{\cap}_{\beta,x,y}(n\in\tau) = \frac{1}{2}\left(1 + \sum_{n\in\mathbb{Z}}P_{\beta,x_1,y_1}(|n|\in\tau)P_{\beta,x_2,y_2}(|n|\in\tau) \right)
\end{equation*}
where $x=(x_1,x_2)$ and $y=(y_1,y_2)$. From Proposition \ref{Fourier} and Parseval's identity we have
\begin{align}\label{Parseval}
&\sum_{n\geq 0} P^{\cap}_{\beta,x,y}(n\in\tau) = \frac{1}{2}\left(1 + <\hat{\phi}^{\texttt{sym}}_{\beta,x_1,y_1},\hat{\phi}^{\texttt{sym}}_{\beta,x_2,y_2}>_{L^2(-\pi,\pi)}\right)\\
& = \frac{1}{2}\left(1 + <2\re((\Id - \hat{\phi}_{\beta})^{-1}_{x_1,y_1})-1,2\re((\Id - \hat{\phi}_{\beta})^{-1}_{x_2,y_2})-1>_{L^2(-\pi,\pi)}\right).
\end{align}
But for all $s,t\in E^q$, 
\[(\Id - \hat{\phi}_{\beta})^{-1}_{s,t} \stackrel{L^2}{\rightarrow} (\Id - \hat{\phi}_0)^{-1}_{s,t}\]
because (\ref{domination}) in Proposition \ref{Fourier} allows the use of the Dominated Convergence Theorem. As a consequence the scalar product in the right-hand side of (\ref{Parseval}) is continuous at $\beta=0$.

We now deal with the case $1\leq l \leq q$. Let us write $\tau^{(2)}-l := \{\tau_n^{(2)}-l, n\geq0\}$. Then
\begin{equation*}
\sum_{l=1}^{\infty}  \delta_n^{(1)} \delta_{n+l}^{(2)} = |\tau^{(1)} \cap (\tau^{(2)}-l) |
\end{equation*}
and $\tau^{(1)} \cap (\tau^{(2)}-l) $ is a delayed renewal process with the same interarrival time distribution as $\tau^{(1)} \cap \tau^{(2)}$, which is the case $l=0$.
\end{proof}

\subsubsection{Control of the second moment}

We will prove in this section

\begin{pr}\label{second_moment}
 There exists $\beta_0>0$ such that for all $\beta\leq\beta_0$,
\begin{equation*}
 \sup_N \mathbb{E}\left(Z_{N,\beta,h_c^a(\beta)}^2\right) < \infty.
\end{equation*}
\end{pr}

\begin{proof}
 Replica method gives:
\begin{align*}
  \mathbb{E} (Z_{N,\beta,h})^2 &= E^{\otimes2}\left( e^{h\sum_{n=1}^N(\delta_n^{(1)}+\delta_n^{(2)})} \mathbb{E}(e^{\beta\sum_{n=1}^N \omega_n (\delta_n^{(1)}+\delta_n^{(2)}) })\right)\\
 & =  E^{\otimes2} \left( e^{\sum_{i=1,2}\sum_{n=1}^N h\delta_n^{(i)}+\frac{\beta^2}{2}\var(\sum \omega_n \delta_n^{(i)})} e^{\beta^2 \cov(\sum \omega_n \delta_n^{(1)},\sum \omega_n \delta_n^{(2)})}\right)
\end{align*}
and
\begin{align*}
 \cov\left(\sum_{n=1}^N \omega_n \delta_n^{(1)},\sum_{n=1}^N \omega_n \delta_n^{(2)}\right) &= \sum_{n,m=1}^N \cov(\omega_n,\omega_m) \delta_n^{(1)} \delta_m^{(2)}\\
& = \sum_{n=1}^N \delta_n^{(1)}\delta_n^{(2)} + \sum_{1\leq|n-m|\leq q} \cov(\omega_n,\omega_m) \delta_n^{(1)} \delta_m^{(2)}.
\end{align*}
Define
\begin{equation*}
C_N(\tau^{(1)},\tau^{(2)}) = \sum_{n=1}^N \delta_n^{(1)}\delta_n^{(2)} + \sum_{k=1}^q \rho_k \sum_{n=1}^N (\delta_n^{(1)}\delta_{n+k}^{(2)} + \delta_n^{(2)}\delta_{n+k}^{(1)})
\end{equation*}
and take $h = h_c^a(\beta)$. Then
\begin{equation*}
 \mathbb{E} (Z_{N,\beta,h_c^a})^2 \leq C_{\beta} E_{\beta}^{\otimes2}\left( e^{\beta^2 C_N(\tau^{(1)},\tau^{(2)}) } \right).
\end{equation*}
Moreover, notice that
\begin{equation*}
C_N(\tau^{(1)},\tau^{(2)}) \leq \sum_{n=1}^{\infty} \delta_n^{(1)}\delta_n^{(2)} + \sum_{k=1}^q \rho_q^+ \sum_{n=1}^{\infty} (\delta_n^{(1)}\delta_{n+k}^{(2)} + \delta_n^{(2)}\delta_{n+k}^{(1)})
\end{equation*}
where $\rho_q^+ = \rho_q\vee0$. By repeated use of Holder's inequality we prove that for all $\beta>0$, there exists nonnegative constants $C_{\beta}, c_0, c_1, \ldots, c_q$ and $e_0, e_1, \ldots, e_q$ such that
\begin{equation*}
 \mathbb{E} (Z_{N,\beta,h_c^a})^2 \leq C_{\beta} \prod_{k=0}^q \left( E_{\beta}^{\otimes2}\left( e^{c_k \beta^2 \sum_{n=1}^{\infty} \delta_n^{(1)} \delta_{n+k}^{(2)}} \right) \right) ^{e_k}.
\end{equation*}
We conclude by using Proposition \ref{intersection}.
\end{proof}

\subsubsection{End of $\limsup$ part}

We define:
\begin{equation*}
 A_N^{\gamma} = \{|\tau \cap [0,N]|\leq N^{\gamma}\}
\end{equation*}
(sometimes we will omit the superscript $\gamma$). We will prove 
\begin{pr}\label{minoration}
 For all $\beta\leq\beta_0$ (with $\beta_0$ as in Proposition \ref{second_moment}), for all $\gamma<\alpha$, there exists $c>0$ such that
\begin{equation*}
 \liminf_N \mathbb{P}(P_{N,\beta,h_c^a}(\overline{A_N^{\gamma}})>c)>c.
\end{equation*}
\end{pr}

Once this is proved, the following proposition provides the limsup part of Theorem \ref{result}:
\begin{pr}\label{limsup}
 If for all $\gamma < \alpha$, there exists some positive constant $c$ such that
\begin{equation}\label{nbr_ren}
 \liminf_N \mathbb{P}(P_{N,\beta,h_c^a}(\overline{A_N^{\gamma}})>c)>c
\end{equation} then
\begin{equation*}
\limsup_{h\rightarrow h_c^a(\beta)^+} \frac{\log(F(\beta,h))}{\log(h-h_c^a(\beta))} \leq \frac{1}{\alpha}.
\end{equation*}
\end{pr}
Proposition \ref{limsup} is proved in \cite{lacoin_irrelevance}. One can check that the independance assumption is not needed there.

To prove Proposition \ref{minoration}, we need the control of the second moment (see Proposition \ref{second_moment}) and several lemmas, such as:

\begin{lem}[Paley-Zygmund inequality]\label{PZ}
 If $Z$ is a nonnegative random variable with finite variance, and if $0<u<1$, then
\begin{equation*}
 \mathbb{P}(Z\geq u\mathbb{E}(Z)) \geq (1-u)^2 \frac{(\mathbb{E}(Z))^2}{\mathbb{E}(Z^2)}.
\end{equation*}
\end{lem}

\begin{lem}\label{mean}
 For all $\beta>0$, $\liminf_{N\rightarrow +\infty} \mathbb{E}Z_{N,\beta,h_c^a}\geq c(\beta) > 0$.
\end{lem}

\begin{proof}
In the first part of the paper it was proved that 
\begin{align*}
\mathbb{E} Z_{N,\beta,h_c^a(\beta)} \geq c_1(\beta) P_{\beta}(n\in\tau)
\end{align*}
when considering the constraint partition function. With free partition function, it is not difficult to prove that
\begin{equation*}
\mathbb{E} Z_{N,\beta,h_c^a(\beta)} \geq c_1(\beta).
\end{equation*}
\end{proof}

\begin{lem}\label{minorZ}
 For all $\beta\leq\beta_0$ ($\beta_0$ as in Proposition \ref{second_moment}), there exists $\delta >0$, $c\in (0,1)$ such that
\begin{equation*}
\inf_N \mathbb{P}(Z_{N,\beta,h_c^a}\geq \delta)>c.
\end{equation*}
\end{lem}

\begin{proof}
 From Lemma \ref{PZ}, we have for all $u\in(0,1)$,
\begin{equation*}
 \mathbb{P}(Z_{N,\beta,h_c^a}\geq u\mathbb{E}(Z_{N,\beta,h_c^a})) \geq (1-u)^2 \frac{(\mathbb{E}(Z_{N,\beta,h_c^a}))^2}{\mathbb{E}(Z_{N,\beta,h_c^a}^2)}
\end{equation*}
For $N$ large enough and $\beta\leq\beta_0$, we have from Lemma \ref{mean} and Lemma \ref{second_moment},
\begin{align*}
 \mathbb{P}(Z_{N,\beta,h_c^a}\geq \frac{u c(\beta)}{2}) &\geq \mathbb{P}(Z_{N,\beta,h_c^a}\geq u\mathbb{E}(Z_{N,\beta,h_c^a}))\\
& \geq (1-u)^2 \frac{(c(\beta)/2)^2}{\sup_N \mathbb{E}(Z_{N,\beta,h_c^a}^2)},
\end{align*}
which is positive thanks to Proposition \ref{second_moment}. The result follows by choosing $u$ close enough to 1.
\end{proof}

\begin{lem}\label{aux}
 For all $\beta>0$, there exists $C_{\beta}>0$ such that
\begin{equation*}
 \mathbb{E}Z_{N,\beta,h_c^a}P_{N,\beta,h_c^a}(A_N^{\gamma}) \leq C_{\beta} P_{\beta}(A_N^{\gamma}).
\end{equation*}
\end{lem}

\begin{proof}
\begin{align*}
 \mathbb{E}Z_{N,\beta,h_c^a}P_{N,\beta,h_c^a}(A_N^{\gamma}) &= \mathbb{E}E\left(\mathbf{1}_{A_N} e^{\sum (\beta \omega_n + h_c^a(\beta)) \delta_n} \right)\\
 & = E\left(\mathbf{1}_{A_N} e^{(-\log\lambda(\beta)-\frac{\beta^2}{2})\sum \delta_n + \frac{\beta^2}{2}\var{\sum \omega_n \delta_n}} \right)\\
& \leq C(\beta) E\left( \mathbf{1}_{A_N} e^{-\log \lambda(\beta)\sum \delta_n + \sum \delta_n \sum_{k=1}^q \rho_k \beta^2 \delta_{n+k}}  \right)\\
& \leq C'(\beta) P_{\beta}(A_N).
\end{align*}
\end{proof}

\begin{lem}\label{aux2}
 For all $\beta >0$, $\gamma<\alpha$,
\begin{equation*}
 P_{\beta}(A_N^{\gamma}) \stackrel{N\rightarrow +\infty}{\longrightarrow} 0.
\end{equation*}
\end{lem}

\begin{proof}
Let us choose arbitrarily $z$ in $E^q$. Then:
 \begin{align*}
  P_{\beta}(A_N^{\gamma}) &= P_{\beta}(|\tau \cap [0,N]|\leq N^{\gamma})\\
& \leq P_{\beta}(|\tau_z \cap [0,N]|\leq N^{\gamma})
 \end{align*}
which tends to $0$ as $N$ tends to $+\infty$ because the tail exponent of the return times of $\tau_z$ is $\alpha$ (see proof of Proposition \ref{varphi_finite}, equation (\ref{tail})), so $|\tau_z \cap [0,N]|$ is of the order of $N^{\alpha}$.
\end{proof}

\begin{proof}[Proof of Proposition \ref{minoration}]
 We first prove that for all $a\in (0,1)$,
\begin{equation*}
 \mathbb{P}(P_{N,\beta,h_c^a}(\overline{A_N})>a) \geq \mathbb{E}P_{N,\beta,h_c^a}(\overline{A_N}) - a.
\end{equation*}
Indeed, this follows from
\begin{equation*}
 \mathbb{P}(P_{N,\beta,h_c^a}(\overline{A_N})>a) \geq \mathbb{E}P_{N,\beta,h_c^a}(\overline{A_N})\mathbf{1}_{\{P_{N,\beta,h_c^a}(\overline{A_N})>a\}}
\end{equation*}
and
\begin{equation*}
 \mathbb{E}P_{N,\beta,h_c^a}(\overline{A_N}) \leq a + \mathbb{E}P_{N,\beta,h_c^a}(\overline{A_N})\mathbf{1}_{\{P_{N,\beta,h_c^a}(\overline{A_N})>a\}}
\end{equation*}
Then, from Lemma \ref{aux},
\begin{align*}
 \mathbb{E}P_{N,\beta,h_c^a}(A_N) &\leq \mathbb{E}\left(P_{N,\beta,h_c^a}(A_N)\mathbf{1}_{\{Z_{N,\beta,h_c^a}\geq \delta\}}\right) + \mathbb{P}(Z_{N,\beta,h_c^a}<\delta)\\
&\leq \delta^{-1}\mathbb{E}Z_{N,\beta,h_c^a}P_{N,\beta,h_c^a}(A_N) + \left(1 -  \mathbb{P}(Z_{N,\beta,h_c^a}\geq\delta)\right)\\
& \leq \delta^{-1}C_{\beta}P_{\beta}(A_N) +  \left(1 -  \mathbb{P}(Z_{N,\beta,h_c^a}\geq\delta)\right).
\end{align*}
From Lemma \ref{aux2} and Lemma \ref{minorZ}, we have
\begin{equation*}
 \liminf_N  \mathbb{E}P_{N,\beta,h_c^a}(\overline{A_N}) \geq c
\end{equation*}
so
\begin{equation*}
 \liminf_N \mathbb{P}(P_{N,\beta,h_c^a}(\overline{A_n})>a) \geq c - a
\end{equation*}
and we conclude the proof by choosing $a$ in $(0,c)$.
\end{proof}

\subsection{Conclusion: proof of Theorem \ref{result}}

We can now conclude:

\begin{proof}[Proof of Theorem \ref{result}]
The bound (\ref{ann_bound}) and Proposition \ref{ann_expo} tell us that $h_c(\beta)\geq h_c^a(\beta)$ and
\[\liminf_{h\rightarrow h_c^a(\beta)^+} \frac{\log F(\beta,h)}{\log(h-h_c^a(\beta))} \geq \frac{1}{\alpha}\]
whereas Proposition \ref{minoration} and Proposition \ref{limsup} tell us that
\[\limsup_{h\rightarrow h_c^a(\beta)^+} \frac{\log F(\beta,h)}{\log(h-h_c^a(\beta))} \leq \frac{1}{\alpha}\] (and so that $h_c(\beta)\leq h_c^a(\beta)$). Therefore we have all the ingredients to prove the theorem.
\end{proof}

\section*{Acknowledgements}
The author would like to thank Giambattista Giacomin for an enlightening discussion. 

\bibliographystyle{spmpsci}      
\bibliography{references}   

\def\polhk\#1{\setbox0=\hbox{\#1}{{\o}oalign{\hidewidth
  {\l}ower1.5ex\hbox{`}\hidewidth\crcr\unhbox0}}}
\begin{thebibliography}{10}
\providecommand{\url}[1]{{#1}}
\providecommand{\urlprefix}{URL }
\expandafter\ifx\csname urlstyle\endcsname\relax
  \providecommand{\doi}[1]{DOI~\discretionary{}{}{}#1}\else
  \providecommand{\doi}{DOI~\discretionary{}{}{}\begingroup
  \urlstyle{rm}\Url}\fi

\bibitem{Alexander_Sido}
Alexander, K.S.: The effect of disorder on polymer depinning transitions.
\newblock Comm. Math. Phys. \textbf{279}(1), 117--146 (2008).
\newblock \doi{10.1007/s00220-008-0425-5}.
\newblock \urlprefix\url{http://dx.doi.org/10.1007/s00220-008-0425-5}

\bibitem{UnzipDNA}
Allahverdyan, A.E., Gevorkian, Z.S., Hu, C.K., Wu, M.C.: Unzipping of dna with
  correlated base sequence.
\newblock Phys. Rev. E \textbf{69}(6), 061,908 (2004).
\newblock \doi{10.1103/PhysRevE.69.061908}

\bibitem{Asmussen}
Asmussen, S.: Applied probability and queues, \emph{Applications of Mathematics
  (New York)}, vol.~51, second edn.
\newblock Springer-Verlag, New York (2003).
\newblock Stochastic Modelling and Applied Probability

\bibitem{CaGiZa07}
Caravenna, F., Giacomin, G., Zambotti, L.: {A renewal theory approach to
  periodic copolymers with adsorption.}
\newblock Ann. Appl. Probab. \textbf{17}(4), 1362--1398 (2007).
\newblock \doi{10.1214/105051607000000159}

\bibitem{Chen}
Chen, X.Y., Bao, L.J., Mo, J.Y., Wang, Y.: Characterizing long-range
  correlation properties in nucleotide sequences.
\newblock Chinese Chemical Letters Vol. 14 \textbf{14}(5), 503--504 (2003)

\bibitem{CFG}
Cornfeld, I.P., Fomin, S.V., Sina{\u\i}, Y.G.: Ergodic theory,
  \emph{Grundlehren der Mathematischen Wissenschaften [Fundamental Principles
  of Mathematical Sciences]}, vol. 245.
\newblock Springer-Verlag, New York (1982)

\bibitem{Doney}
Doney, R.A.: One-sided local large deviation and renewal theorems in the case
  of infinite mean.
\newblock Probab. Theory Related Fields \textbf{107}(4), 451--465 (1997).
\newblock \doi{10.1007/s004400050093}.
\newblock \urlprefix\url{http://dx.doi.org/10.1007/s004400050093}

\bibitem{GarsiaLamperti}
Garsia, A., Lamperti, J.: A discrete renewal theorem with infinite mean.
\newblock Comment. Math. Helv. \textbf{37}, 221--234 (1962/1963)

\bibitem{GG_Overview}
Giacomin", G.: Renewal sequences, disordered potentials and pinning phenomena.
\newblock In: Spin Glasses: Statics and Dynamics

\bibitem{GG_Book}
Giacomin, G.: Random polymer models.
\newblock Imperial College Press, London (2007)

\bibitem{Hammersley}
Hammersley, J.M.: Generalization of the fundamental theorem on sub-additive
  functions.
\newblock Proc. Cambridge Philos. Soc. \textbf{58}, 235--238 (1962)

\bibitem{bubble}
Jeon, J.H., Park, P.J., Sung, W.: The effect of sequence correlation on bubble
  statistics in double-stranded dna.
\newblock The Journal of Chemical Physics \textbf{125} (2006)

\bibitem{lacoin_irrelevance}
Lacoin, H.: The martingale approach to disorder irrelevance for pinning models.
\newblock Electronic Communications in Probability \textbf{15}, 418--427
  (2010).
\newblock
  \urlprefix\url{http://www.math.washington.edu/~ejpecp/EcpVol15/paper38.abs.h%
tml}

\bibitem{Peng}
Peng, C.K., Buldyrev, S.V., Goldberger, A.L., Havlin, S., Sciortino, F.,
  Simons, M., Stanley, H.E.: Long-range correlations in nucleotide sequences.
\newblock Nature \textbf{356}, 168--170 (1992)

\bibitem{Seneta}
Seneta, E.: Non-negative matrices and {M}arkov chains.
\newblock Springer Series in Statistics. Springer, New York (2006)

\bibitem{Spitzer_PRW}
Spitzer, F.: Principles of random walks, second edn.
\newblock Springer-Verlag, New York (1976).
\newblock Graduate Texts in Mathematics, Vol. 34

\bibitem{Steele}
Steele, J.M.: Kingman's subadditive ergodic theorem.
\newblock Ann. Inst. H. Poincar{\'e} Probab. Statist. \textbf{25}(1), 93--98
  (1989)

\bibitem{Fabio_replica}
Toninelli, F.L.: {A replica-coupling approach to disordered pinning models.}
\newblock Commun. Math. Phys. \textbf{280}(2), 389--401 (2008).
\newblock \doi{10.1007/s00220-008-0469-6}

\bibitem{Toninelli_Survey}
Toninelli, F.L.: Localization transition in disordered pinning models.
\newblock In: Methods of Contemporary Mathematical Statistical Physics, Lecture
  Notes in Mathematics, pp. 129--176 (2009)

\end{thebibliography}

\end{document}